\documentclass[11pt]{article}

\usepackage{amsmath}
\usepackage{amsthm}
\usepackage{amssymb}
\usepackage{color}
\usepackage{url}
\usepackage{graphicx}
\usepackage{ascmac}
\usepackage{float}
\usepackage{comment}
\usepackage{ifthen}
\usepackage{mathdots}
\usepackage{algorithm}
\usepackage{algorithmic}
\usepackage{mathrsfs}

\theoremstyle{definition}
\newtheorem{thm}{Theorem}
\newtheorem{prop}[thm]{Proposition}
\newtheorem{lem}[thm]{Lemma}

\newtheorem{dfn}[thm]{Definition}

\newtheorem{rem}[thm]{Remark}
\newtheorem{assump}[thm]{Assumption}

\allowdisplaybreaks[4]

\numberwithin{equation}{section}
\numberwithin{thm}{section}
\numberwithin{thm_ja}{section}

\usepackage{typearea}
\typearea{12}

\newcommand{\newconvex}{transport convex}
\newcommand{\newconvexity}{transport convexity}
\newcommand{\newsmooth}{transport $L$-smooth}
\newcommand{\newsmoothness}{transport $L$-smoothness}

\title{Accelerated gradient descent method for functionals of probability measures by new convexity and smoothness based on transport maps}
\author{Ken'ichiro Tanaka%
\footnote{
Department of Mathematical Informatics, Graduate School of Information Science and Technology, The University of Tokyo, 7-3-1 Hongo, Bunkyo, Tokyo 113-8656, Japan. 
\texttt{kenichiro@mist.i.u-tokyo.ac.jp}}
\footnote{
PRESTO Japan Science and Technological Agency (JST), Japan.}}
\date{\today}

\begin{document}

\maketitle

\begin{abstract}
We consider problems of minimizing functionals $\mathcal{F}$ of probability measures on the Euclidean space.  
To propose an accelerated gradient descent algorithm for such problems, 
we consider gradient flow of transport maps that give push-forward measures of an initial measure. 
Then we propose a deterministic accelerated algorithm 
by extending Nesterov's acceleration technique with momentum. 
This algorithm do not based on the Wasserstein geometry. 
Furthermore, 
to estimate the convergence rate of the accelerated algorithm, 
we introduce new convexity and smoothness for $\mathcal{F}$ based on transport maps. 
As a result, 
we can show that the accelerated algorithm converges faster than a normal gradient descent algorithm. 
Numerical experiments support this theoretical result. 
 
\end{abstract}

\tableofcontents

\section{Introduction}

In this paper, 
we consider problems of minimizing functionals of probability measures on the $d$-dimensional Euclidean space $\mathbf{R}^{d}$.  
Let 
$\mathcal{P}_{2}(\mathbf{R}^{d})$ 
be the set of probability measures on $\mathbf{R}^{d}$ whose second-order moments are finite
and let 
$\mathcal{F}: \mathcal{P}_{2}(\mathbf{R}^{d}) \to \mathbf{R}$ 
be a functional of a measure in $\mathcal{P}_{2}(\mathbf{R}^{d})$. 
Our objective is to solve a minimization problem of $\mathcal{F}$:
\begin{align}
\text{minimize} \quad \mathcal{F}(\mu) \quad \text{subject to} \quad \mu \in \mathcal{P}_{2}(\mathbf{R}^{d}). 
\label{eq:def_min_problem}
\end{align}
We assume that there exists an optimal measure
$\mu_{\ast} \in \mathcal{P}_{2}(\mathbf{R}^{d})$. 
To give algorithms solving Problem \eqref{eq:def_min_problem}, 
we consider gradient flow for $\mathcal{F}$. 
Our main aim is to propose efficient algorithms for getting empirical measures 
that make $\mathcal{F}$ as small as possible. 

Optimization in measure spaces like~\eqref{eq:def_min_problem} 
is a fundamental problem in statistics and other various areas. 
In statistics, estimate of probability density functions is ubiquitous \cite{silverman1986density}
and has a similar form to~\eqref{eq:def_min_problem}. 
Besides, 
such optimization is considered in 
optimal control \cite{casas2012approximation}, 
problems of equilibrium measures \cite{borodachov2019discrete, saff2013logarithmic}, 
numerical integration in RKHSs \cite{10.5555/3020652.3020694, Oettershagen:2017, 10.1214/18-STS683, kanagawa2016convergence}, 
training of neural networks \cite{chizat2022meanfield, pmlr-v151-nitanda22a, chizat2018global, nitanda2017stochastic}, etc. 
In these areas, 
we often encounter problems of finding an empirical measure that approximates a given measure. 
Typical examples are
sampling problems from a given distribution in statistics 
and
problems of constructing numerical integration formulas in numerical analysis.
In these problems
we require an accurate approximation of a given measure based on some criterion. 
Therefore
these problems can be regarded as restriction of $\mu$ to empirical measures in problem~\eqref{eq:def_min_problem}. 

For such problems, 
many methods based on the Monte Carlo method have been devised so far. 
In recent years, 
especially in the field of machine learning, 
application of mathematical optimization methods to problem~\eqref{eq:def_min_problem}
has been studied. 
In particular, 
gradient descent methods for functionals of probability measures have attracted much attention 
\cite{chizat2018global, arbel2019maximum, daneshmand2023efficient, hertrich2023wasserstein, nitanda2017stochastic, chewi2020svgd}.
These methods are often based on 
the Wasserstein gradient flow
\cite{ambrosio2021lectures, ambrosio2005gradient, jordan1998variational, peyre2019computational}, 
the gradient flow based on the geometrical structure given by the Wasserstein distance known in the field of optimal transport
\cite{villani2009optimal, villani2021topics}. 
Studies have shown the convergence of those gradient methods for functionals $\mathcal{F}$ with displacement convexity 
\cite{mccann1997convexity}. 

For mathematical optimization on Euclidean spaces or manifolds, 
there are known methods to accelerate gradient methods.
For convex optimization problems, 
Nesterov 
\cite{nesterov1983method, nesterov2018lectures}
devises an acceleration method using momentum, whose effectiveness is well known. 
For Nesterov's acceleration method, 
convergence analysis based on a corresponding differential equation
has been performed 
\cite{su2014differential, su2016differential, wilson2021lyapunov}.

In view of this fact, 
we propose an accelerated gradient method for problem~\eqref{eq:def_min_problem} for functionals $\mathcal{F}$ in this paper. 
To this end, 
we extend Nesterov's acceleration technique with momentum to a gradient method for problem~\eqref{eq:def_min_problem}. 
For this extension, 
we introduce new convexity and $L$-smoothness for $\mathcal{F}$ based on transport maps
in order to facilitate introducing momentum.
As a result, 
we get a fast method 
that is deterministic unlike 
the methods with Langevin dynamics 
\cite{wibisono2018sampling, durmus2019analysis, bernton2018langevin, cheng2018underdamped, salim2020primal, hsieh2018mirrored,
dalalyan2017theoretical, vempala2019rapid, zou2019sampling, ma2019there}. 
Below we 
discuss relation between this study and existing ones
and 
describe the contributions of this study. 

\subsection{Related work}

Gradient descent methods of (empirical) measures has been often studied. 
In particular, 
Stein variational gradient descent (SVGD)
\cite{liu2016stein, liu2017stein, korba2020non, lu2019scaling, nusken2023geometry, chewi2020svgd, nusken2023manyparticle}
has recently studied for the Kullback-Leibler divergence. 
In the context of the Wasserstein gradient flow 
\cite{ambrosio2021lectures, ambrosio2005gradient, jordan1998variational, peyre2019computational}, 
algorithms are considered for functionals with the displacement convexity, 
which is the convexity on constant speed geodesics in a space of probability measures
\cite{arbel2019maximum, daneshmand2023efficient}. 
Several studies propose acceleration methods
\cite{ma2019there, carrillo2019convergence, pmlr-v97-taghvaei19a, pmlr-v97-liu19i, wang2022accelerated}.
Besides, 
several studies 
\cite{chizat2022meanfield, pmlr-v151-nitanda22a}
propose gradient descent methods based on the normal convexity%
\footnote{
The normal convexity means that 
\(
\mathcal{F}(\lambda \mu + (1-\lambda) \nu)
\leq
\lambda \mathcal{F}(\mu) + (1-\lambda) \mathcal{F}(\nu)
\)
holds for $\lambda \in [0,1]$ and $\mu, \nu \in \mathcal{P}_{2}(\mathbf{R}^{d})$.
}.
Note that 
the displacement convexity is based on the convex combination of random variables in $\mathbf{R}^{d}$
whereas 
the normal convexity is based on the convex combination of probability measures in $\mathcal{P}_{2}(\mathbf{R}^{d})$. 

In convex optimization, 
the first-order methods have attracted attention because of their tractability for large-scale problems 
in various areas like machine learning. 
Because the usual gradient descent method, a basic first-order method, 
do not realize very fast convergence, 
its acceleration has been studied.
Nesterov's accelerated gradient method \cite{nesterov2018lectures} is a representative 
and its convergence analysis has been performed 
based on a corresponding differential equation \cite{su2014differential, su2016differential, wilson2021lyapunov}. 

\subsection{Contribution}

We extend Nesterov's acceleration technique with momentum to a gradient method for problem~\eqref{eq:def_min_problem}. 
In view of our main aim to find empirical measures, 
we consider gradient flow of transport maps $X(t, \cdot): \mathbf{R}^{d} \to \mathbf{R}^{d}$ 
that give push-forward measures $X(t,\cdot)_{\#} \mu_{0}$ of an initial measure $\mu_{0}$
because of the following two reasons.
First, 
we can get point clouds that provides empirical measure 
by discretization of such transport maps $X(t, \cdot)$.
Second, 
it is easy to 
facilitate introducing momentum 
that is indispensable for extension of Nesterov's method
if we deal with the gradient flow of the maps $X(t,\cdot): \mathbf{R}^{d} \to \mathbf{R}^{d}$. 
As a result, 
we do not need to use the Wasserstein geometry of $\mathcal{P}_{2}(\mathbf{R}^{d})$. 
Our approach is a different one from \cite{ma2019there, carrillo2019convergence, pmlr-v97-taghvaei19a, pmlr-v97-liu19i, wang2022accelerated} for acceleration. 

To guarantee effectiveness of the extension, 
we assume some convexity and $L$-smoothness for $\mathcal{F}$. 
Note that 
for each $t$ the push-forward measure $X(t,\cdot)_{\#} \mu_{0}$ is the probability measure of the random variable $X(t, X_{0})$, 
where $X_{0}$ is the random variable following $\mu_{0}$. 
Therefore 
we need convexity based on the convex combination of random variables. 
However, 
the displacement convexity is not enough because it is the convexity just on the constant speed geodesics. 
Taking this situation into account, 
we introduce new convexity for $\mathcal{F}$ based on transport maps. 
In addition, 
we introduce new $L$-smoothness for $\mathcal{F}$ to analyze discretizations of the gradient flow and its accelerated one. 
As a result, 
we show that the extended Nesterov's acceleration is effective for functionals $\mathcal{F}$ with the new convexity and smoothness. 
For this analysis, 
we rely on the method of the Lyapunov function \cite{wilson2021lyapunov}. 
The range of functionals $\mathcal{F}$ with these new properties
is narrower than that with the displacement convexity. 
However, as shown in this paper, typical functionals are shown to have the new properties.


\begin{rem}
There is another framework for getting empirical measures based on some functionals $\mathcal{F}$. 
In the field of kernel quadrature, 
kernel herding algorithms are studied 
\cite{welling2009herding, chen2010super, bach2012equivalence, tsuji2022pairwise}. 
But those algorithms 
are based on the conditional gradient method and 
take points one by one among prepared candidate points in $\mathbf{R}^{d}$. 
They are not based on gradient flows of transport maps that are continuous counterparts of point clouds. 
\end{rem}

\subsection{Organization of this paper}

In Section~\ref{sec:math_prelim}, 
we describe mathematical preliminaries including typical examples of functionals. 
Furthermore, 
we define gradients of functionals based on transport maps and present their examples. 
In Section~\ref{sec:grad_flow}, 
we describe the gradient flow of transport maps, its discretization, and their accelerated versions.  
In Section~\ref{sec:new_conv_and_L_smooth}, 
we introduce new convexity and $L$-smoothness for $\mathcal{F}$ based on transport maps. 
In Section~\ref{sec:grad_Lyap_2}, 
we prove theorems about convergence rates of the gradient flows described in Section~\ref{sec:grad_flow}. 
In Section~\ref{sec:num_ex}, 
we conduct numerical experiments. 
In Section~\ref{sec:concl}, 
we conclude this paper.

\section{Mathematical preliminaries}
\label{sec:math_prelim}

\subsection{Notation}

In this paper, 
$\| \cdot \|$ and $\langle \cdot, \cdot \rangle$ are the Euclidean norm and inner product in $\mathbf{R}^{d}$, 
respectively. 
The function space
$L^{2}(\mathbf{R}^{d} \to \mathbf{R}^{d})$ 
is the set of square integrable maps from $\mathbf{R}^{d}$ to $\mathbf{R}^{d}$. 

For a measure $\mu \in \mathcal{P}_{2}(\mathbf{R}^{d})$
and a function $f: \mathbf{R}^{d} \to \mathbf{R}$, 
the domain of integration of $f$ with respect to $\mu$ is $\mathbf{R}^{d}$
unless otherwise stated. 
The push-forward of $\mu$ by a map $F: \mathbf{R}^{d} \to \mathbf{R}^{d}$ 
is denoted by $F_{\#}\mu$, i.e., $(F_{\#}\mu)(B) := \mu(F^{-1}(B))$ for any measurable set $B \subset \mathbf{R}^{d}$. 
Then, for a function $f: \mathbf{R}^{d} \to \mathbf{R}$, the following holds:
\(
\int f(x)\, \mathrm{d}(F_{\#}\mu)(x)
= 
\int f(F(y)) \, \mathrm{d} \mu_{0}(y).
\)

When we consider functions with two variables in $\mathbf{R}^{d}$, 
we use $\nabla_{1}$ and $\nabla_{2}$ as the gradient operators with respect to the first and second variables, respectively. 

\subsection{Functional of a measure}
\label{sec:func_meas}

We consider the following functionals as examples of $\mathcal{F}$. 
They are typical examples treated in \cite[Chap.~9]{ambrosio2005gradient} and \cite[Lect.~15]{ambrosio2021lectures}. 

\subsubsection{Internal energy}
\label{sec:func_internal_ene}

Let 
$U: [0, \infty) \to (-\infty, \infty]$
be a convex $\mathrm{C}^{1}$ function with
\begin{align}
U(0) = 0
\quad \text{and} \quad 
\lim_{x \to +0} x \, U'(x) = 0. 
\label{eq:diff_U_assump}
\end{align}
For $\mu \in \mathcal{P}_{2}(\mathbf{R}^{d})$, 
we define $\mathcal{U}(\mu)$ by
\begin{align}
\mathcal{U}(\mu)
:=
\begin{cases}
\displaystyle \int U(\rho(x)) \, \mathrm{d}x & (\mu = \rho \mathcal{L}^{d}), \\
+\infty & (\text{otherwise}), 
\end{cases}
\label{eq:def_U_int_ene}
\end{align}
where $\rho$ is the density function with respect to the Lebesgue measure $\mathcal{L}^{d}$ of an absolutely continuous measure $\mu$. 

\begin{rem}
\label{rem:entropy}
A typical example of the internal energy is the negative entropy, 
which is given by~\eqref{eq:def_U_int_ene} with the function $U(x) = x \log x$.
Here we adopt the standard rule $U(0) = 0$. 
\end{rem}

\subsubsection{Potential}
\label{sec:func_pot}

Let $V: \mathbf{R}^{d} \to [0, \infty)$ be a $\mathrm{C}^{1}$ function. 
For $\mu \in \mathcal{P}_{2}(\mathbf{R}^{d})$, 
we define $\mathcal{V}(\mu)$ by
\begin{align}
\mathcal{V}(\mu) 
:= 
\int V(x) \, \mathrm{d}\mu(x).
\label{eq:def_of_F_V_pot_en}
\end{align}

\begin{rem}
\label{rem:KL_div}
The Kullback-Leibler divergence is expressed by the sum of the internal energy and potential. 
Let $\nu = \sigma \mathcal{L}^{d}$ be the probability measure with a density function $\sigma(x)$. 
We assume that $0 < \sigma(x) \leq C$ for some constant $C \in (0, \infty)$.
Then the Kullback-Leibler divergence from $\nu$ to the probability measure $\mu = \rho \mathcal{L}^{d}$ with a density function $\rho(x)$ is given by
\begin{align}
\int \rho(x) \log \frac{\rho(x)}{\sigma(x)} \, \mathrm{d}x
& = 
\int U(\rho(x)) \, \mathrm{d}x + \int V(x) \rho(x) \, \mathrm{d}x - \log C
\notag
\end{align}
where $U(x) = x \log x$ and $V(x) = - \log (\sigma(x)/C)$. 
\end{rem}

\subsubsection{Interaction energy}
\label{sec:func_inter}

Let $W: \mathbf{R}^{d} \times \mathbf{R}^{d} \to [0, \infty)$ 
be a $\mathrm{C}^{1}$ function with symmetry: $W(x,y) = W(y,x)$ for any $x, y \in \mathbf{R}^{d}$. 
For $\mu \in \mathcal{P}_{2}(\mathbf{R}^{d})$, 
we define $\mathcal{W}$ by
\begin{align}
\mathcal{W}(\mu) 
:= 
\frac{1}{2}
\int \int W(x,y) \, \mathrm{d}\mu(x) \mathrm{d}\mu(y). 
\label{eq:def_of_F_W_interact_en}
\end{align}
This functional is regarded as interaction energy. 

\subsection{Gradient of a functional}

To consider gradient flow for a functional $\mathcal{F}$, 
we consider its gradient based on a flow map. 

\begin{dfn}
\label{dfn:grad_of_functional}
Let 
$\mu_{0} \in \mathcal{P}_{2}(\mathbf{R}^{d})$
be a probability measure  
and let
$X: [0,\infty) \times \mathbf{R}^{d} \to \mathbf{R}^{d}$
be a map that is differentiable with respect to the first variable. 
For $t \in [0, \infty)$, 
let
$\mu_{t}$ be the push-forward of $\mu_{0}$ by $X(t, \cdot)$, i.e., 
\(
\mu_{t} := X(t, \cdot)_{\#} \mu_{0}.
\)
For a functional 
$\mathcal{F}: \mathcal{P}_{2}(\mathbf{R}^{d}) \to \mathbf{R}$, 
if there exists a function
$\mathcal{G}_{\mathcal{F}}[\cdot](\cdot): \mathcal{P}_{2}(\mathbf{R}^{d}) \times \mathbf{R}^{d} \to \mathbf{R}^{d}$
satisfying
\begin{align}
\frac{\mathrm{d}}{\mathrm{d}t} \mathcal{F}(\mu_{t})
=
\int 
\left \langle \mathcal{G}_{\mathcal{F}}[\mu_{t}](X(t,x)), \frac{\partial}{\partial t} X(t,x) \right \rangle
\mathrm{d}\mu_{0}(x), 
\label{eq:char_vec_field_for_F_mu_t}
\end{align}
we call $\mathcal{G}_{\mathcal{F}}[\cdot](\cdot)$ a gradient of $\mathcal{F}$. 
\end{dfn}

When an explicit form of $\mathcal{F}$ is given, 
we can derive its gradient $\mathcal{G}_{\mathcal{F}}$ explicitly. 
Gradients of the functionals $\mathcal{U}$,  $\mathcal{V}$, and $\mathcal{W}$ 
are given by the following propositions, 
whose proofs are written in Section~\ref{sec:proofs_gradient}. 

\begin{prop}[Gradient of $\mathcal{U}$]
\label{thm:grad_U}
Let 
$U: [0, \infty) \to (-\infty, \infty]$
be a convex $\mathrm{C}^{1}$ function satisfying~\eqref{eq:diff_U_assump}. 
Let $\mu \in \mathcal{P}(\mathbf{R}^{d})$ be an absolutely continuous measure with a compact support and a density $\rho$, 
and let 
$F: \mathbf{R}^{d} \to \mathbf{R}^{d}$ be a differentiable map whose Jacobian $\nabla F(x)$ is positive definite for any $x \in \mathbf{R}^{d}$. 
Then, 
the function $\mathcal{G}_{\mathcal{U}}$ characterized by
\begin{align}
\mathcal{G}_{\mathcal{U}}[F_{\#}\mu](F(x)) 
= 
\frac{1}{\rho(x)}
\mathop{\mathrm{div}} 
\left[
\left\{
U'\left( \frac{\rho(x)}{|\nabla F(x)|} \right) \rho(x) - U\left( \frac{\rho(x)}{|\nabla F(x)|} \right) |\nabla F(x)|
\right\}
\nabla F(x)^{-1}
\right]
\label{eq:def_grad_for_U_internal}
\end{align}
is a gradient of the functional $\mathcal{U}$ in~\eqref{eq:def_U_int_ene}. 
Here the operator $\mathrm{div}$ gives the column-wise divergence of a matrix, i.e., 
\(
\mathop{\mathrm{div}} A(x) 
=
(\mathop{\mathrm{div}} a_{1}(x), \ldots, \mathop{\mathrm{div}} a_{d}(x))^{\top}
\)
for a matrix $A(x) = (a_{1}(x), \ldots, a_{d}(x))$.
In addition, 
$|B|$ is the determinant of a square matrix $B$.
\end{prop}

\begin{prop}[Gradient of $\mathcal{V}$]
\label{thm:grad_V}
Let $V: \mathbf{R}^{d} \to [0, \infty)$ be a $\mathrm{C}^{1}$ function. 
The function $\mathcal{G}_{\mathcal{V}}$ given by
\begin{align}
\mathcal{G}_{\mathcal{V}}[\mu](x) = \nabla V(x)
\label{eq:def_grad_for_F_V_pot_en}
\end{align}
is a gradient of the functional $\mathcal{V}$ in~\eqref{eq:def_of_F_V_pot_en}.
\end{prop}

\begin{prop}[Gradient of $\mathcal{W}$]
\label{thm:grad_W}
Let $W: \mathbf{R}^{d} \times \mathbf{R}^{d} \to [0, \infty)$ 
be a $\mathrm{C}^{1}$ function with $W(x,y) = W(y,x)$ for any $x, y \in \mathbf{R}^{d}$. 
The function $\mathcal{G}_{\mathcal{W}}$ given by
\begin{align}
\mathcal{G}_{\mathcal{W}}[\mu](x) 
= 
\int \nabla_{1} W(x, y) \, \mathrm{d}\mu(y) 
\label{eq:def_grad_for_F_W_inter}
\end{align}
is a gradient of the functional $W$ in~\eqref{eq:def_of_F_W_interact_en}. 
\end{prop}

\section{Gradient flow}
\label{sec:grad_flow}

\subsection{Gradient flow and its discretization}
\label{sec:grad_flow_and_disc}

To solve Problem~\eqref{eq:def_min_problem}, 
we begin with finding a curve of measures $\mu_{t} \in \mathcal{P}_{2}(\mathbf{R}^{2})$ 
such that $\mathcal{F}(\mu_{t})$ is a decreasing function of $t$. 
For this purpose, 
we take a measure $\mu_{0} \in \mathcal{P}_{2}(\mathbf{R}^{2})$
and 
find a flow map $X(t, \cdot)$ such that $\mu_{t} := X(t, \cdot)_{\#} \mu_{0}$ makes $\mathcal{F}(\mu_{t})$ decreasing. 
Such a flow map can be found via equations given by the gradient $\mathcal{G}_{\mathcal{F}}$ satisfying~\eqref{eq:char_vec_field_for_F_mu_t}. 

We consider a differential equation determining a flow map $X: [0,\infty) \times \mathbf{R}^{d} \to \mathbf{R}^{d}$:
\begin{align}
& \frac{\partial}{\partial t} X(t, x) = -\mathcal{G}_{\mathcal{F}}[X(t, \cdot)_{\#} \mu_{0}](X(t, x)), 
\label{eq:ODE_for_X} \\
& X(0, x) = x,
\label{eq:ODE_for_X_init}
\end{align}
where $x \in \mathbf{R}^{d}$.
We assume that this equation has a unique global solution $X(t,x)$ for each $x \in \mathbf{R}^{d}$.
Then,  define $\mu_{t}$ by 
\begin{align}
\mu_{t} := X(t,\cdot)_{\#} \mu_{0}.
\label{eq:meas_curve_of_grad_flow}
\end{align}


The following proposition shows that 
$\mu_{t}$ in \eqref{eq:meas_curve_of_grad_flow} 
gives a curve that decreases $\mathcal{F}(\mu_{t})$. 

\begin{prop}
\label{prop:decrease_of_F_mu_t}
Let $X(t, \cdot)$ be a solution of the equation \eqref{eq:ODE_for_X}--\eqref{eq:ODE_for_X_init}. 
Then, $\mu_{t}$ in \eqref{eq:meas_curve_of_grad_flow} satisfies 
\[
\displaystyle
\frac{\mathrm{d}}{\mathrm{d}t}\mathcal{F}(\mu_{t}) 
= 
- \int_{\mathbf{R}^{d}} \left \| \mathcal{G}_{\mathcal{F}}[\mu_{t}](X(t, x)) \right \|^{2} \, \mathrm{d}\mu_{0}(x).
\]
\end{prop}

\begin{proof}
By substituting~\eqref{eq:ODE_for_X} into~\eqref{eq:char_vec_field_for_F_mu_t}, 
we have the conclusion. 
\end{proof}


To give algorithms based on the gradient flow, 
we need to discretize the time $t$ of the equation~\eqref{eq:ODE_for_X}--\eqref{eq:ODE_for_X_init}. 
By applying the explicit Euler method to the equation, 
we get the following discrete gradient flow:
\begin{align}
& X_{n+1}(x) = X_{n}(x) - \gamma \, \mathcal{G}_{\mathcal{F}}[(X_{n})_{\#} \mu_{0}](X_{n}(x)), 
\label{eq:disc_grad_flow_for_X} \\
& X_{0}(x) = x, 
\label{eq:meas_hat_mu_n}
\end{align}
where $\gamma > 0$ corresponds to a step size $\varDelta t$ of time. 
By using a solution $\{ X_{n} \}$ of this system, 
we define $\hat{\mu}_{n}$ by 
\begin{align}
\hat{\mu}_{n} := (X_{n})_{\#} \mu_{0}.
\label{eq:disc_PF_of_mu0_by_X}
\end{align}

\subsection{Accelerated gradient flow and its discretization}
\label{sec:acc_grad_flow_and_disc}

For a given vector field $\mathcal{G}_{\mathcal{F}}$,
a flow map $X$ is given by the equation \eqref{eq:ODE_for_X}--\eqref{eq:ODE_for_X_init}. 
Then, 
a curve $\{ \mu_{t} \}_{t \in [0,\infty)} \subset \mathcal{P}_{2}(\mathbf{R}^{d})$
is given by the flow map via \eqref{eq:meas_curve_of_grad_flow}. 
Taking this procedure into account, 
we expect that we can obtain a curve  
making $\mathcal{F}$ decrease more rapidly 
by ``accelerating'' the equation \eqref{eq:ODE_for_X}--\eqref{eq:ODE_for_X_init} for $X$. 

In the field of optimization of convex functions $f: \mathbf{R}^{d} \to \mathbf{R}^{d}$, 
algorithms based on gradient flow and their acceleration have been developed. 
In particular, 
Nesterov's accelerated gradient (NAG) flow is well-known. 
As shown below, 
we can apply the technique of the NAG flow to the equation \eqref{eq:ODE_for_X}--\eqref{eq:ODE_for_X_init} for $X$. 
Consequently, 
we can show faster convergence of $\mathcal{F}(\mu_{t}^{\mathrm{ac}})$ for $\mu_{t}^{\mathrm{ac}}$ given by an accelerated flow map $X^{\mathrm{ac}}$. 

 
We consider the following system of equations:
\begin{align}
& \frac{\partial}{\partial t} X^{\mathrm{ac}}(t, x) = \frac{r}{t} (Y^{\mathrm{ac}}(t,x) - X^{\mathrm{ac}}(t,x)), 
\label{eq:AccODE_for_XY_1_moment} \\
& \frac{\partial}{\partial t} Y^{\mathrm{ac}}(t, x) = - rt^{r-1} \, \mathcal{G}_{\mathcal{F}}[X^{\mathrm{ac}}(t, \cdot)_{\#} \mu_{0}](X^{\mathrm{ac}}(t, x)),
\label{eq:AccODE_for_XY_2_vecfield} \\
& X^{\mathrm{ac}}(0, x) = Y^{\mathrm{ac}}(0,x ) = x,
\label{eq:AccODE_for_XY_init}
\end{align}
where $r$ is a real number with $r \geq 2$. 
By using a solution $(X^{\mathrm{ac}},Y^{\mathrm{ac}})$ of this system, 
we define $\mu_{t}^{\mathrm{ac}}$ and $\nu_{t}^{\mathrm{ac}}$ by 
\begin{align}
& \mu_{t}^{\mathrm{ac}} := X^{\mathrm{ac}}(t, \cdot)_{\#} \mu_{0}, \quad \text{and}
\label{eq:PF_of_mu0_by_X_Acc} \\
& \nu_{t}^{\mathrm{ac}} := Y^{\mathrm{ac}}(t, \cdot)_{\#} \mu_{0}, 
\label{eq:PF_of_mu0_by_Y_Acc} 
\end{align}
respectively. 


To give algorithms based on the accelerated gradient flow, 
we consider a discrete counterpart of the system in \eqref{eq:AccODE_for_XY_1_moment}--\eqref{eq:AccODE_for_XY_init}. 
More precisely, 
we consider the following discrete gradient flow: 
\begin{align}
& X_{n+1}^{\mathrm{ac}}(x) = Y_{n}^{\mathrm{ac}}(x) + \frac{a_{n+1} - a_{n}}{a_{n+1}} \, (Z_{n}^{\mathrm{ac}}(x) - Y_{n}^{\mathrm{ac}}(x)), 
\label{eq:acc_disc_grad_flow_for_X} \\
& Z_{n+1}^{\mathrm{ac}}(x) = Z_{n}^{\mathrm{ac}}(x) - (a_{n+1} - a_{n}) \, \mathcal{G}_{\mathcal{F}}[(X_{n+1}^{\mathrm{ac}})_{\#} \mu_{0}](X_{n+1}^{\mathrm{ac}}(x)), 
\label{eq:acc_disc_grad_flow_for_Z} \\
& Y_{n+1}^{\mathrm{ac}}(x) = \mathfrak{F}[X_{n+1}^{\mathrm{ac}}, Z_{n}^{\mathrm{ac}}, Y_{n}^{\mathrm{ac}}](x), 
\label{eq:acc_disc_grad_flow_for_Y} \\
& X_{0}^{\mathrm{ac}}(x) = Y_{0}^{\mathrm{ac}}(x) = Z_{0}^{\mathrm{ac}}(x) = x,
\label{eq:disc_AccODE_for_XYZ_init}
\end{align}
where $\mathfrak{F}$ and $\{ a_{n} \} \subset \mathbf{R}$ are a functional and an increasing sequence, respectively. 
By using a solution $\{ (X_{n}^{\mathrm{ac}}, Y_{n}^{\mathrm{ac}}, Z_{n}^{\mathrm{ac}}) \}$ of this system, 
we define $\hat{\mu}_{n}$, $\hat{\nu}_{n}$, and $\hat{\xi}_{n}$ by 
\begin{align}
& \hat{\mu}_{n}^{\mathrm{ac}} := (X_{n}^{\mathrm{ac}})_{\#} \mu_{0}, 
\label{eq:disc_PF_of_mu0_by_X_Acc} \\
& \hat{\nu}_{n}^{\mathrm{ac}} := (Y_{n}^{\mathrm{ac}})_{\#} \mu_{0}, \quad \text{and} 
\label{eq:disc_PF_of_mu0_by_Y_Acc} \\
& \hat{\xi}_{n}^{\mathrm{ac}} := (Z_{n}^{\mathrm{ac}})_{\#} \mu_{0}, 
\label{eq:disc_PF_of_mu0_by_Z_Acc}
\end{align}
respectively. 
In Section \ref{sec:est_acc_grad}, 
we define $\mathfrak{F}$ and $\{ a_{n} \}$ so that $\mathcal{F}(\hat{\nu}_{n}^{\mathrm{ac}})$
converges to the optimal value faster than $\mathcal{F}(\hat{\mu}_{n})$ for $\hat{\mu}_{n}$ in~\eqref{eq:disc_PF_of_mu0_by_X}.

\begin{rem}
The relation between 
the system in \eqref{eq:AccODE_for_XY_1_moment}--\eqref{eq:AccODE_for_XY_init} 
and 
its discrete counterpart in~\eqref{eq:acc_disc_grad_flow_for_X}--\eqref{eq:disc_AccODE_for_XYZ_init}
is not straightforward. 
\end{rem}

\section{New convexity and $L$-smoothess of functionals}
\label{sec:new_conv_and_L_smooth}

The gradient flows and their discrete counterparts in Sections \ref{sec:grad_flow_and_disc} and \ref{sec:acc_grad_flow_and_disc}
do not necessarily guarantee convergence of $\mathcal{F}(\mu_{t})$ and $\mathcal{F}(\hat{\mu}_{n})$ to the optimal value $\mathcal{F}(\mu_{\ast})$. 
To guarantee the convergence, 
we introduce notions of convexity and $L$-smoothness for functionals $\mathcal{F}$. 
The $L$-smoothness is used for showing the convergence of the discrete flows. 

\subsection{Convexity}

\begin{dfn}
\label{dfn:pf_convex}
Let $\mathcal{F}: \mathcal{P}_{2}(\mathbf{R}^{d}) \to \mathbf{R}$ be a functional 
and 
$\mathcal{G}_{\mathcal{F}}$ be a gradient of $\mathcal{F}$ satisfying \eqref{eq:char_vec_field_for_F_mu_t}. 
The functional $\mathcal{F}$ is
\emph{\newconvex} if it satisfies 
\begin{align}
\mathcal{F}(T_{\#} \mu) - \mathcal{F}(S_{\#} \mu)
\geq 
\int  
\left \langle 
\mathcal{G}_{\mathcal{F}}[S_{\#} \mu](S(x)), \, 
T(x) - S(x)
\right \rangle
\, \mathrm{d}\mu(x)
\label{eq:def_pf_convex}
\end{align}
for any $S, T \in L^{2}(\mathbf{R}^{d} \to \mathbf{R}^{d})$ and $\mu \in \mathcal{P}_{2}(\mathbf{R}^{d})$. 
\end{dfn}

Below we show that 
the functionals $\mathcal{U}$, $\mathcal{V}$, and $\mathcal{W}$ defined in Section~\ref{sec:func_meas}
have the {\newconvexity} 
under some conditions. 
We begin with the functional $\mathcal{U}$ in Section~\ref{sec:func_internal_ene}.
In this case, 
we restrict ourselves to maps $S$ and $T$ with positive definite Jacobians. 

\begin{thm}[Convexity of $\mathcal{U}$]
\label{thm:conv_of_U}
Let 
$U: [0, \infty) \to (-\infty, \infty]$
be a convex $\mathrm{C}^{1}$ function satisfying~\eqref{eq:diff_U_assump}. 
Assume that the function $\tilde{U}_{d}: (0,\infty) \to \mathbf{R}$ defined by
\begin{align}
\tilde{U}_{d}(z)
:= 
U(z^{-d}) \, z^{d}
\label{eq:def_tilde_U}
\end{align}
is convex and non-increasing. 
In addition, 
assume that 
differentiable maps $S, T \in L^{2}(\mathbf{R}^{d} \to \mathbf{R}^{d})$ 
have positive definite Jacobians $\nabla S(x)$ and $\nabla T(x)$ for any $x$.
Then,
the gradient $\mathcal{G}_{\mathcal{U}}$ satisfies~\eqref{eq:def_pf_convex} for $\mathcal{F} = \mathcal{U}$. 
\end{thm}

Next, we show the {\newconvexity} of the functional $\mathcal{V}$ in Section~\ref{sec:func_pot}.

\begin{thm}[Convexity of $\mathcal{V}$]
\label{thm:conv_of_V}
Assume that a $\mathrm{C}^{1}$ function $V: \mathbf{R}^{d} \to \mathbf{R}$ is convex. 
Then, the functional $\mathcal{V}$ in~\eqref{eq:def_of_F_V_pot_en} is {\newconvex}, 
where the gradient $\mathcal{G}_{\mathcal{V}}$ is defined by~\eqref{eq:def_grad_for_F_V_pot_en}. 
\end{thm}

Finally, 
we consider the functional $\mathcal{W}$ given in Section~\ref{sec:func_inter}
in the special case that the function $W(x,y)$ is given by $W(x,y) = \tilde{W}(x - y)$ 
with a function $\tilde{W}: \mathbf{R}^{d} \to \mathbf{R}$. 
Note that this $W$ is symmetric 
if $\tilde{W}(-x) = \tilde{W}(x)$ holds for any $x \in \mathbf{R}^{d}$. 

\begin{thm}[Convexity of $\mathcal{W}$]
\label{thm:conv_of_W}
\label{thm:func_inter_pf_convex}
Let $\tilde{W}: \mathbf{R}^{d} \to \mathbf{R}$ be a $\mathrm{C}^{1}$ convex function satisfying $\tilde{W}(-x) = \tilde{W}(x)$. 
Then, the functional $\mathcal{W}$ in~\eqref{eq:def_of_F_W_interact_en} with $W(x,y) = \tilde{W}(x - y)$ is {\newconvex}, 
where the gradient $\mathcal{G}_{\mathcal{W}}$ is defined by~\eqref{eq:def_grad_for_F_W_inter}. 
\end{thm}

Proofs of these theorems are written in Section~\ref{sec:proofs_convex_smooth}.
Basically they are based on those in \cite[Lect.~15]{ambrosio2021lectures}. 
However we precisely describe them 
because our new convexity is different from that in \cite[Lect.~15]{ambrosio2021lectures} (the displacement convexity). 

\subsection{$L$-smoothess}

\begin{dfn}
\label{dfn:pf_L_smooth}
Let $\mathcal{F}: \mathcal{P}_{2}(\mathbf{R}^{d}) \to \mathbf{R}$ be a functional 
and 
$\mathcal{G}_{\mathcal{F}}$ be a gradient of $\mathcal{F}$ satisfying \eqref{eq:char_vec_field_for_F_mu_t}. 
Let $L$ be a positive real number. 
The functional $\mathcal{F}$ is \emph{\newsmooth} if it satisfies 
\begin{align}
& \mathcal{F}(T_{\#} \mu) - \mathcal{F}(S_{\#} \mu) 
\notag \\
& \leq 
\int  
\left \langle
\mathcal{G}_{\mathcal{F}}[S_{\#}\mu](S(x)), \, T(x) - S(x)
\right \rangle
\mathrm{d} \mu(x)
+ 
\frac{L}{2} \int 
\| T(x) - S(x) \|^{2}
\, \mathrm{d} \mu(x)
\label{eq:def_pf_L_smooth}
\end{align}
for any $S, T \in L^{2}(\mathbf{R}^{d} \to \mathbf{R}^{d})$ and $\mu \in \mathcal{P}_{2}(\mathbf{R}^{d})$. 
\end{dfn}

We have not proved the {\newsmoothness} of 
the functional $\mathcal{U}$ in Section~\ref{sec:func_internal_ene}. 
We leave this proof for future work and just show a partial assertion in Theorem~\ref{thm:func_U_restricted_L_smooth} below. 
In this theorem we prove inequality~\eqref{eq:def_pf_L_smooth} for $\mathcal{U}$ 
under several restrictive assumptions for $\mu$, $U$, $S$, and $T$, 
which are presented in Section~\ref{sec:proofs_smooth} later.
In addition, 
we present some observations about these assumptions in Remark~\ref{rem:L_smooth_U_assump} in Section~\ref{sec:proofs_smooth}. 

\begin{thm}
\label{thm:func_U_restricted_L_smooth}
Let 
$U: [0, \infty) \to (-\infty, \infty]$
be a convex $\mathrm{C}^{2}$ function satisfying~\eqref{eq:diff_U_assump} and
\begin{align}
\lim_{x \to +0} x^{2} \, U''(x) = 0.
\label{eq:diff_U_assump_2nd}
\end{align}
Let $\rho$ be the density of $\mu \in \mathcal{P}_{2}(\mathbf{R}^{d})$. 
Under Assumptions~\ref{assump:bounded_C1_C2}--\ref{assump:bounded_C4_div} in Section~\ref{sec:proofs_smooth}
for $\rho$, $U$, $S$, and $T$, 
inequality~\eqref{eq:def_pf_L_smooth} for $\mathcal{F} = \mathcal{U}$ holds for some constant $L \geq 0$. 
\end{thm}

Below we show that 
the functionals $\mathcal{V}$ and $\mathcal{W}$ defined in Section~\ref{sec:func_meas}
have the {\newsmoothness} 
under some conditions. 
First, 
we show the {\newsmoothness} of the functional $\mathcal{V}$ in \eqref{eq:def_of_F_V_pot_en}
under an assumption about a function $V$. 

\begin{thm}[$L$-smoothness of $\mathcal{V}$]
\label{thm:func_pot_L_smooth}
Assume that a function $V: \mathbf{R}^{d} \to \mathbf{R}$ is $L_{V}$-smooth, i.e., 
\begin{align}
V(y) - V(x) 
\leq 
\left \langle \nabla V(x), y - x \right \rangle
+ \frac{L_{V}}{2} \| y - x \|^{2}. 
\label{eq:V_L_smooth}
\end{align}
Then 
the functional $\mathcal{V}$ in~\eqref{eq:def_of_F_V_pot_en} for $V$
is {\newsmooth} with $L = L_{V}$. 
\end{thm}

Next, 
to deal with the functional $\mathcal{W}$ in~\eqref{eq:def_of_F_W_interact_en}, 
we define the matrix $H_{W}(x,y)$ by
\begin{align}
H_{W}(x,y)
:=
\begin{bmatrix}
\nabla_{1}^{2} W(x,y) & \nabla_{1} \nabla_{2} W(x,y) \\
\nabla_{2} \nabla_{1} W(x,y) & \nabla_{2}^{2} W(x,y)
\end{bmatrix}
\label{eq:def_W_Hess_mat}
\end{align}
for a $\mathrm{C}^{2}$ function 
$W: \mathbf{R}^{d} \times \mathbf{R}^{d} \to \mathbf{R}$. 
We can show the {\newsmoothness} of $\mathcal{W}$
under an assumption for $H_{W}(x,y)$. 

\begin{thm}[$L$-smoothness of $\mathcal{W}$]
\label{thm:func_inter_L_smooth}
For a symmetric $\mathrm{C}^{2}$ function $W: \mathbf{R}^{d} \times \mathbf{R}^{d} \to \mathbf{R}$, 
let $H_{W}(x,y)$ be the matrix defined by~\eqref{eq:def_W_Hess_mat}. 
Assume that 
\begin{align}
r_{W}
:=
\sup_{x,y \in \mathbf{R}^{d}}
\sigma
\left(
H_{W}(x,y)
\right)
\label{eq:2diff_W_eigen_finite}
\end{align}
is finite, 
where $\sigma$ is the operator indicating the spectral radius of a square matrix.
Then
the functional $\mathcal{W}$ in~\eqref{eq:def_of_F_W_interact_en} 
is {\newsmooth} with $L = r_{W}$. 
\end{thm}

Proofs of these theorems are written in Section~\ref{sec:proofs_convex_smooth}.


\section{Estimate of the convergence rates of the gradient flows}
\label{sec:grad_Lyap_2}

We estimate the convergence rates of the gradient flows in Section~\ref{sec:grad_flow}. 
We assume the {\newconvexity} for functionals $\mathcal{F}$ for the continuous flow, 
and the {\newconvexity} and $L$-smoothness for the discrete flow. 

To illustrate a general principle for the estimate, 
we sketch a method for the estimate in the case of the continuous flow
given by~\eqref{eq:ODE_for_X}--\eqref{eq:ODE_for_X_init} in Section~\ref{sec:grad_flow_and_disc} 
as an example. 
It is based on the method of the Lyapunov function in \cite{wilson2021lyapunov}.
We begin with defining a function $\mathcal{L}(t)$ by 
\begin{align}
\mathcal{L}(t) 
:= 
t \, (\mathcal{F}(\mu_{t}) - \mathcal{F}(\mu_{\ast}))
+ 
D(\mu_{t}, \mu_{\ast}), 
\label{eq:Lyap_for_C_flow_pre}
\end{align}
where $D$ is a non-negative functional measuring a discrepancy between $\mu_{t}$ and $\mu_{\ast}$. 
We call $\mathcal{L}(t)$ a Lyapunov function. 
Then,
by using the equation \eqref{eq:ODE_for_X}--\eqref{eq:ODE_for_X_init} and {\newconvexity} of $\mathcal{F}$, 
we prove that $\mathcal{L}(t)$ is monotone decreasing. 
As a result, 
we have
\begin{align}
t \, (\mathcal{F}(\mu_{t}) - \mathcal{F}(\mu_{\ast}))
\leq 
\mathcal{L}(t) 
\leq 
\mathcal{L}(0)
\quad
\left( 
= D(\mu_{0}, \mu_{\ast})
\right).
\label{eq:Lyap_grad_flow_upper_bound}
\end{align}
These inequalities imply that
the convergence rate of $\mathcal{F}(\mu_{t}) - \mathcal{F}(\mu_{\ast})$ is $\mathrm{O}(1/t)$. 

To realize the above method, 
we need to introduce a discrepancy functional $D$ that is consistent with the {\newconvexity}. 
Suppose that 
maps $S$ and $T$ in $L^{2}(\mathbf{R}^{d} \to \mathbf{R}^{d})$ 
give measures $S_{\#} \mu_{0}$ and $T_{\#} \mu_{0}$ in $\mathcal{P}_{2}(\mathbf{R}^{d})$, respectively. 
Then we define $D$ for these measures by
\begin{align}
D(S_{\#} \mu_{0}, T_{\#} \mu_{0}) := \frac{1}{2} \int \| S(x) - T(x) \|^{2} \, \mathrm{d}\mu_{0}(x). 
\label{eq:def_discr_func}
\end{align}
In addition,
we suppose that the following assumption is satisfied in the rest of this section. 
\begin{assump}
\label{assump:opt_meas_push_forward}
The optimal measure $\mu_{\ast}$ is given by the push-forward of $\mu_{0}$ by a map $X_{\ast} \in L^{2}(\mathbf{R}^{d} \to \mathbf{R}^{d})$, 
i.e., 
$\mu_{\ast} = (X_{\ast})_{\#} \mu_{0}$.
\end{assump}

Under Assumption~\ref{assump:opt_meas_push_forward}, 
we show the convergence rate $\mathrm{O}(1/t)$ for the continuous flow in Theorem~\ref{thm:c_rate_cont_flow} below. 
In the other cases in 
Section~\ref{sec:grad_flow}, 
similar methods can be used for the estimate. 
We show convergence rates for these cases in Theorems~\ref{thm:c_rate_disc_flow}--\ref{thm:c_rate_acc_disc_flow}. 
Proofs of these theorems are given by Section~\ref{sec:proofs_conv_rate}.

\subsection{Convergence rates of the gradient flows}
\label{sec:est_grad}

Convergence rates for the cases in Section~\ref{sec:grad_flow_and_disc}
are given by the following theorems. 

\begin{thm}
\label{thm:c_rate_cont_flow}
Let $X(t,x)$ be the flow map given by the gradient flow \eqref{eq:ODE_for_X}--\eqref{eq:ODE_for_X_init}
and 
let $\mu_{t}$ be the measure given by \eqref{eq:meas_curve_of_grad_flow}. 
Assume that $\mathcal{F}$ is {\newconvex} and 
Assumption~\ref{assump:opt_meas_push_forward} holds. 
Then 
we have
$\mathcal{F}(\mu_{t}) - \mathcal{F}(\mu_{\ast}) = \mathrm{O}(1/t)$ 
as $t \to \infty$. 
\end{thm}

\begin{thm}
\label{thm:c_rate_disc_flow}
Let $X_{n}(x)$ be the flow map given by the gradient flow \eqref{eq:disc_grad_flow_for_X}--\eqref{eq:meas_hat_mu_n}
and 
let $\hat{\mu}_{n}$ be the measure given by \eqref{eq:disc_PF_of_mu0_by_X}. 
Assume that $\mathcal{F}$ is {\newconvex} and {\newsmooth}. 
Furthermore, 
assume that $\gamma \leq 3/(2 L)$ for $\gamma$ in \eqref{eq:disc_grad_flow_for_X}. 
Then, 
under Assumption~\ref{assump:opt_meas_push_forward}, 
we have
$\mathcal{F}(\hat{\mu}_{n}) - \mathcal{F}(\mu_{\ast}) = \mathrm{O}(1/n)$ 
as $n \to \infty$. 
\end{thm}

\subsection{Convergence rates of the accelerated gradient flows}
\label{sec:est_acc_grad}

Convergence rates for the cases in Section~\ref{sec:acc_grad_flow_and_disc}
are given by the following theorems. 

\begin{thm}
\label{thm:c_rate_acc_cont_flow}
Let $X^{\mathrm{ac}}(t,x)$ and $Y^{\mathrm{ac}}(t,x)$ be the flow maps given 
by the gradient flow \eqref{eq:AccODE_for_XY_1_moment}--\eqref{eq:AccODE_for_XY_init}
with $r > 0$
and 
let $\mu_{t}^{\mathrm{ac}}$ and $\nu_{t}^{\mathrm{ac}}$ be the measures given 
by \eqref{eq:PF_of_mu0_by_X_Acc} and \eqref{eq:PF_of_mu0_by_Y_Acc}, respectively. 
Assume that $\mathcal{F}$ is {\newconvex} and 
Assumption~\ref{assump:opt_meas_push_forward} holds. 
Then 
we have
$\mathcal{F}(\mu_{t}^{\mathrm{ac}}) - \mathcal{F}(\mu_{\ast}) = \mathrm{O}(1/t^{r})$ 
as $t \to \infty$. 
\end{thm}

\begin{thm}
\label{thm:c_rate_acc_disc_flow}
We define
the functional $\mathfrak{F}$
and
the sequence $\{ a_{n} \}$ in 
\eqref{eq:acc_disc_grad_flow_for_X}--\eqref{eq:acc_disc_grad_flow_for_Y}
by 
\begin{align}
& \mathfrak{F}[X, Y, Z](x) := X(x) - \frac{1}{L} \, \mathcal{G}_{\mathcal{F}}[X_{\#} \mu_{0}](X(x)), \quad \text{and}
\label{eq:def_mathcal_G} \\
& a_{n} := \frac{(n+1)(n+2)}{16L},
\label{eq:def_A_n}
\end{align}
respectively. 
Let $X^{\mathrm{ac}}_{n}(x)$, $Y^{\mathrm{ac}}_{n}(x)$, and $Z^{\mathrm{ac}}_{n}(x)$ be the flow maps given 
by the gradient flow 
\eqref{eq:acc_disc_grad_flow_for_X}--\eqref{eq:disc_AccODE_for_XYZ_init}
and 
let $\hat{\mu}_{n}^{\mathrm{ac}}$, $\hat{\nu}_{n}^{\mathrm{ac}}$, and $\hat{\xi}_{n}^{\mathrm{ac}}$ be the measures given 
by \eqref{eq:disc_PF_of_mu0_by_X_Acc}, \eqref{eq:disc_PF_of_mu0_by_Y_Acc}, and \eqref{eq:disc_PF_of_mu0_by_Z_Acc}, respectively. 
Assume that $\mathcal{F}$ is {\newconvex} and {\newsmooth}. 
Then, 
under Assumption~\ref{assump:opt_meas_push_forward}, 
we have
$\mathcal{F}(\hat{\nu}_{n}^{\mathrm{ac}}) - \mathcal{F}(\mu_{\ast}) = \mathrm{O}(1/n^{2})$ 
as $n \to \infty$. 
\end{thm}

\section{Numerical experiments}
\label{sec:num_ex}

We show effectiveness of the proposed acceleration 
for minimization of functionals $\mathcal{F}$ by several examples. 
For this purpose, 
we apply the algorithms given 
by~\eqref{eq:disc_grad_flow_for_X}--\eqref{eq:meas_hat_mu_n} 
and~\eqref{eq:acc_disc_grad_flow_for_X}--\eqref{eq:disc_AccODE_for_XYZ_init}
with~\eqref{eq:def_mathcal_G}--\eqref{eq:def_A_n}
to some functionals%
\footnote{
We conduct experiments just to confirm the effect of our acceleration comparing the normal gradient descent.
Comparison with other methods like the SVGD or the methods with Langevin dynamics is left to future work, 
although their frameworks are slightly different from ours.}. 

For the former algorithm (the normal gradient descent), 
we can use any $\gamma$ with $0 < \gamma \leq 3/(2L)$ as a time step
according to Theorem~\ref{thm:c_rate_disc_flow}. 
Here we note that the
{\newsmoothness}
holds for any $L$ that is larger than the smallest one. 
Therefore
we can use the time step $\gamma = (3/2) \tau$ with 
\begin{align}
\tau = \frac{1}{L}
\label{eq:def_tau}
\end{align}
for any $L$ for which the 
{\newsmoothness}
holds. 
For the latter algorithm (the accelerated gradient descent), 
thinking in the same way as the former one, 
we use $\tau$ in place of $1/L$ in~\eqref{eq:acc_disc_grad_flow_for_X}--\eqref{eq:disc_AccODE_for_XYZ_init} 
and~\eqref{eq:def_mathcal_G}--\eqref{eq:def_A_n}. 
Taking the above observation into account, we use 
\begin{align}
\tau = 0.0025, 0.005, 0.01
\label{eq:used_tau}
\end{align}
in the following experiments.

\subsection{Example 1}
Let $d = 1$ and 
let $\mathcal{F}_{1}$ be defined by
\begin{align}
\mathcal{F}_{1}(\mu)
:=
\mathcal{W}_{1}(\mu) + \mathcal{V}_{1}(\mu),
\label{eq:ex_for_num_comp_1}
\end{align}
where
\begin{align}
& \mathcal{W}_{1}(\mu) 
:= 
\int_{\mathbf{R}} \int_{\mathbf{R}} \| x - y \|^{2} \, \mathrm{d}\mu(x) \, \mathrm{d}\mu(y)
\quad \text{and}
\notag \\
& \mathcal{V}_{1}(\mu) 
:= 
\int_{\mathbf{R}} \| x - 1 \|^{2} \, \mathrm{d}\mu(x). 
\notag 
\end{align}
The functionals $\mathcal{W}_{1}$ and $\mathcal{V}_{1}$ are defined by $W(x,y) = \|x-y\|^{2}$ and $V(x) = \|x-1\|^{2}$, respectively. 
By Theorems~\ref{thm:conv_of_U}--\ref{thm:func_inter_L_smooth}, 
these functionals are 
{\newconvex} and 
{\newsmooth} with $L = 2$. 
Therefore $\mathcal{F}_{1}$ is also {\newconvex} and {\newsmooth} with $L = 4$. 
Then the values of the parameter $\tau$ in~\eqref{eq:used_tau} are feasible. 
In addition, 
the optimal measure for $\mathcal{F}_{1}$ is the Dirac measure $\delta_{1}$. 

Results for $\mathcal{F}_{1}$ are shown in Figure~\ref{fig:one_dim_inter}. 
In some beginning steps, 
the normal algorithm outperforms the accelerated one. 
After those steps, 
however, 
the accelerated algorithm outperforms the normal one
and
the decay rate of the values of $\mathcal{F}_{1}$ given by the accelerated one
is larger than that given by the normal one. 
This result is consistent with Theorems~\ref{thm:c_rate_disc_flow} and~\ref{thm:c_rate_acc_disc_flow}. 
In addition, 
we can observe that the acceleration is more effective with smaller time steps. 

\begin{figure}[ht]
\centering
\begin{minipage}{0.3\linewidth}
\includegraphics[width=\linewidth]{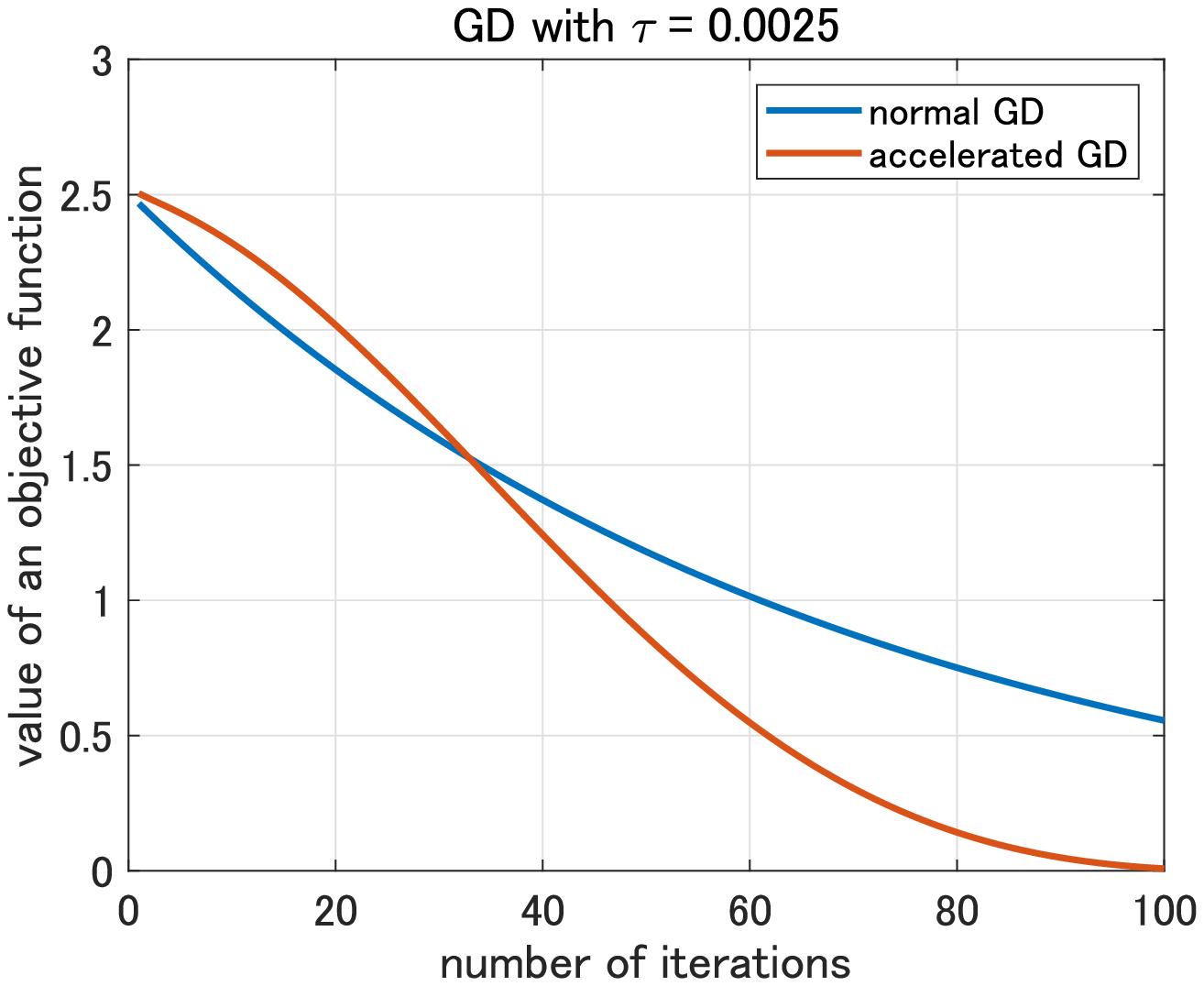}
\end{minipage} \ 
\begin{minipage}{0.3\linewidth}
\includegraphics[width=\linewidth]{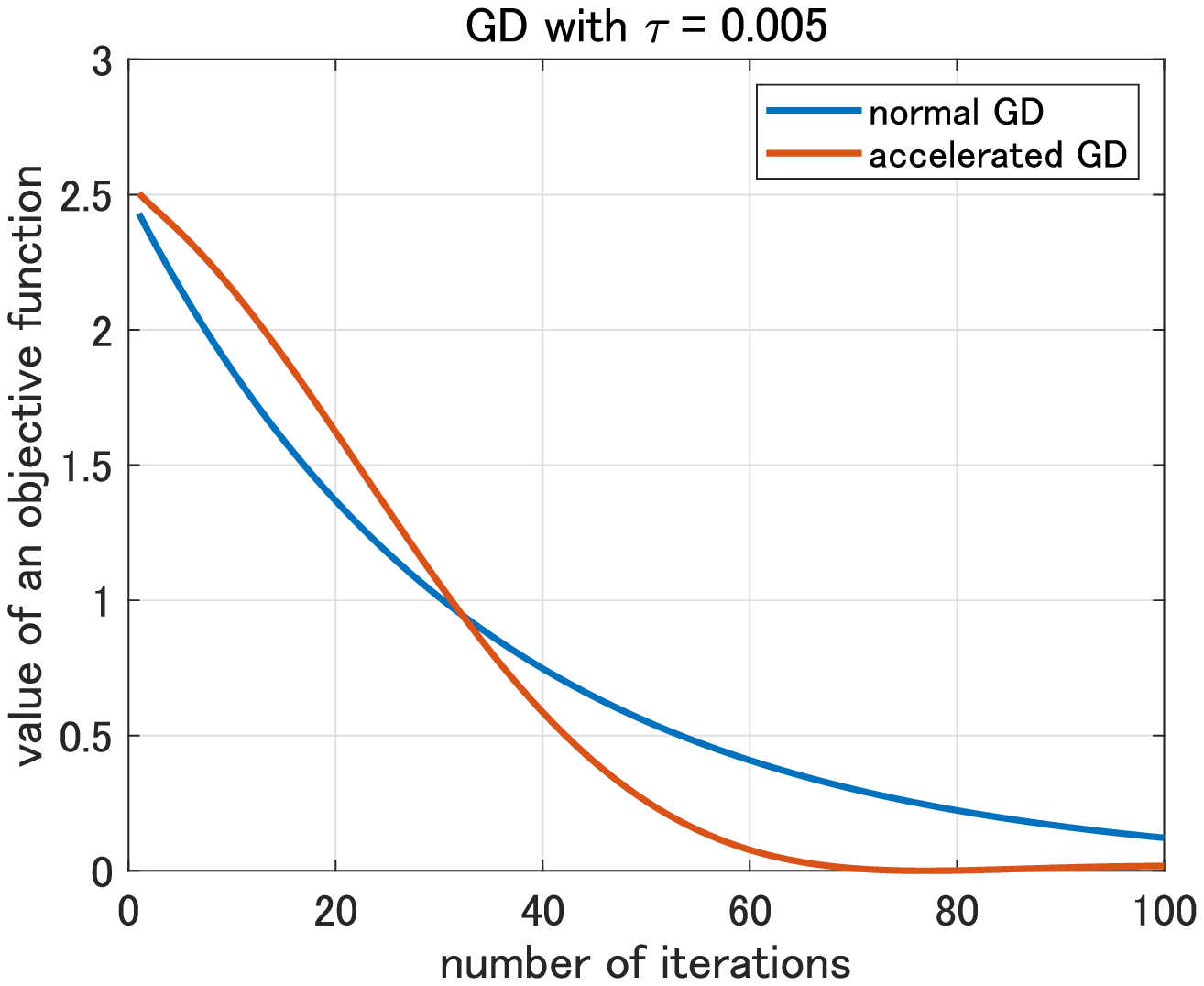}
\end{minipage} \ 
\begin{minipage}{0.3\linewidth}
\includegraphics[width=\linewidth]{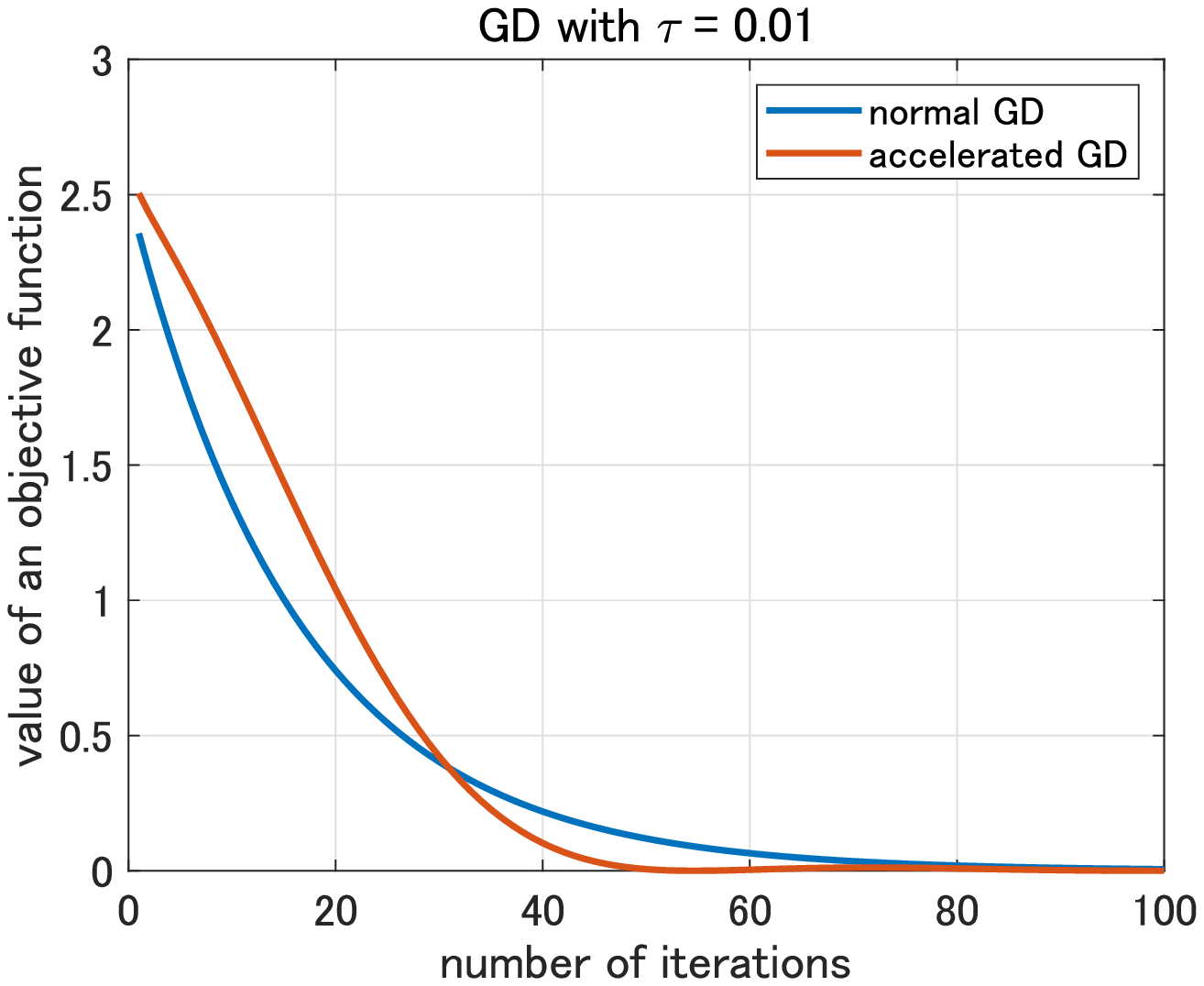}
\end{minipage}
\caption{Results for $\mathcal{F}_{1}$ in \eqref{eq:ex_for_num_comp_1} ($1$-D interaction energy plus a linear function).}
\label{fig:one_dim_inter}
\end{figure}

\subsection{Example 2}
\label{sec:num_ex_2_neg_ent}
Let $d = 2$ and 
let $\mathcal{F}_{2}$ be defined by
\begin{align}
\mathcal{F}_{2}(\mu)
:=
\kappa \, \mathcal{U}_{2}(\mu) + \mathcal{V}_{2}(\mu),
\label{eq:ex_for_num_comp_2}
\end{align}
where 
$\kappa$ is a positive real number, 
$\mathcal{U}_{2}$ is the negative entropy given by the function $U(x) = x \log x$,  
and 
\begin{align}
& \mathcal{V}_{2}(\mu) 
:= 
\int_{\mathbf{R}^{2}} \frac{1}{2} \| x \|^{2} \, \mathrm{d}\mu(x). 
\notag
\end{align}
The functional $\mathcal{V}_{2}$ is defined by $V(x) = \| x \|^{2}/2$. 
By Theorems~\ref{thm:conv_of_U}--\ref{thm:conv_of_W}, 
these functionals are 
{\newconvex}. 
Furthermore, 
$\mathcal{V}_{2}$ is 
{\newsmooth} with $L = 2$ 
by Theorem~\ref{thm:func_pot_L_smooth}. 
However, 
as stated before Theorem~\ref{thm:func_U_restricted_L_smooth}, 
we do not know whether $\mathcal{U}_{2}$ is 
{\newsmooth} or not. 
Therefore we just observe what happens in this experiment in the case of $\mathcal{F}_{2}$. 
It is known that the minimizer of $\mathcal{F}_{2}$ is the measure of the normal distribution 
\cite[Remark 9.12 and Figure 9.8]{peyre2019computational}. 
We use 
\begin{align}
\kappa = 0.05, 0.1
\label{eq:used_kappa}
\end{align}
for this experiment. 
In addition, 
we need to compute the negative entropy $\mathcal{U}_{2}(\mu)$ for empirical measures $\mu$ for this experiment
although this value is $+\infty$ according to the definition in~\eqref{eq:def_U_int_ene}.
To this end, 
we use the well-known $k$-nearest neighbor method \cite[Remark 9.12]{peyre2019computational}
to approximate the negative entropy by a point cloud. 
Therefore we do not compute 
the complicated gradient $\mathcal{G}_{\mathcal{U}_{2}}$ given by \eqref{eq:def_grad_for_U_internal}
in this experiment. 

We show results for $\mathcal{F}_{2}$ in Figure~\ref{fig:two_dim_neg_ent}. 
We can observe a similar phenomenon 
that the accelerated algorithm outperforms the normal one after some beginning steps. 
In addition, 
smaller time steps are preferable in that the final states are stable. 

\begin{figure}[ht]
\centering
\begin{minipage}{0.3\linewidth}
\includegraphics[width=\linewidth]{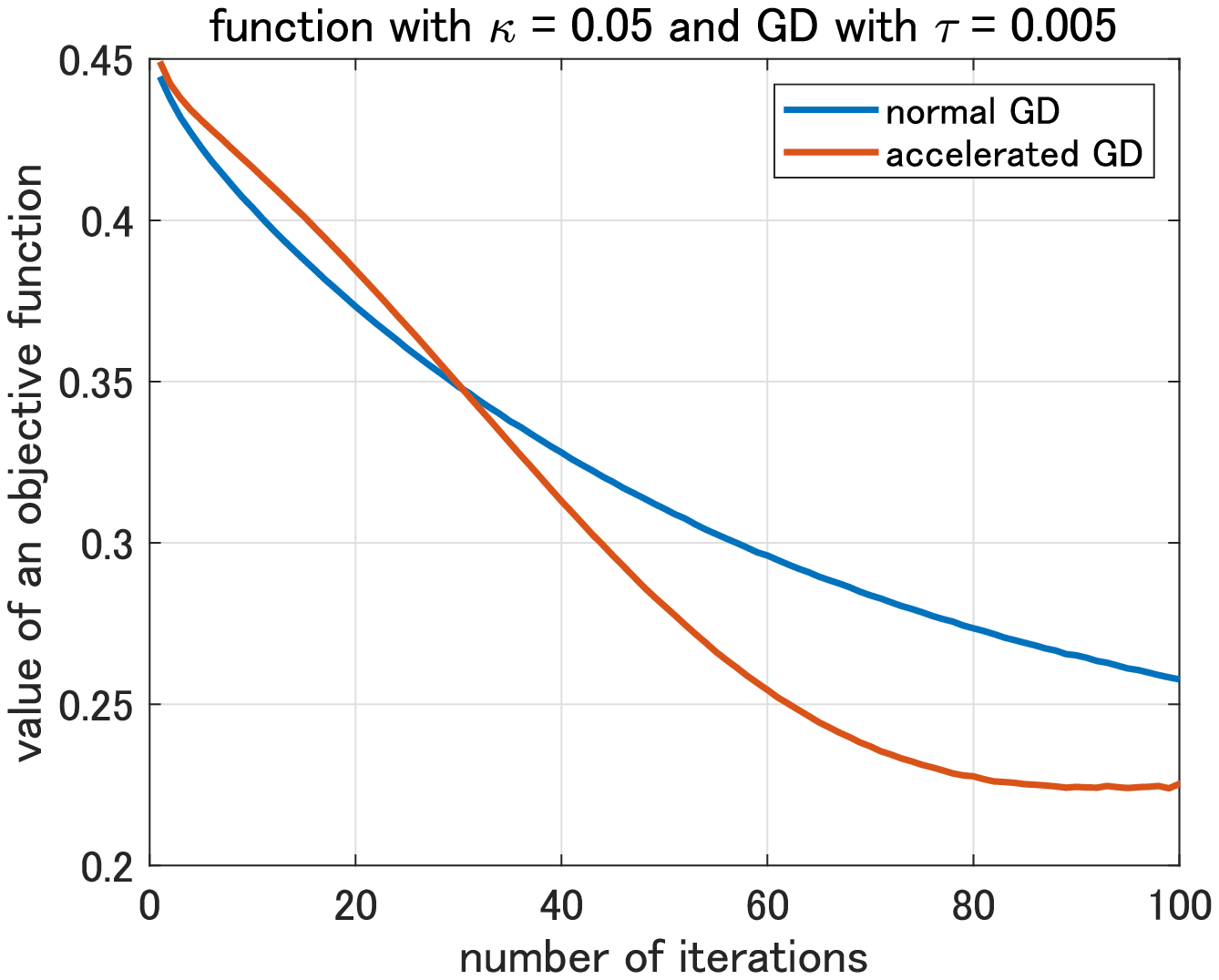}
\end{minipage} \ 
\begin{minipage}{0.3\linewidth}
\includegraphics[width=\linewidth]{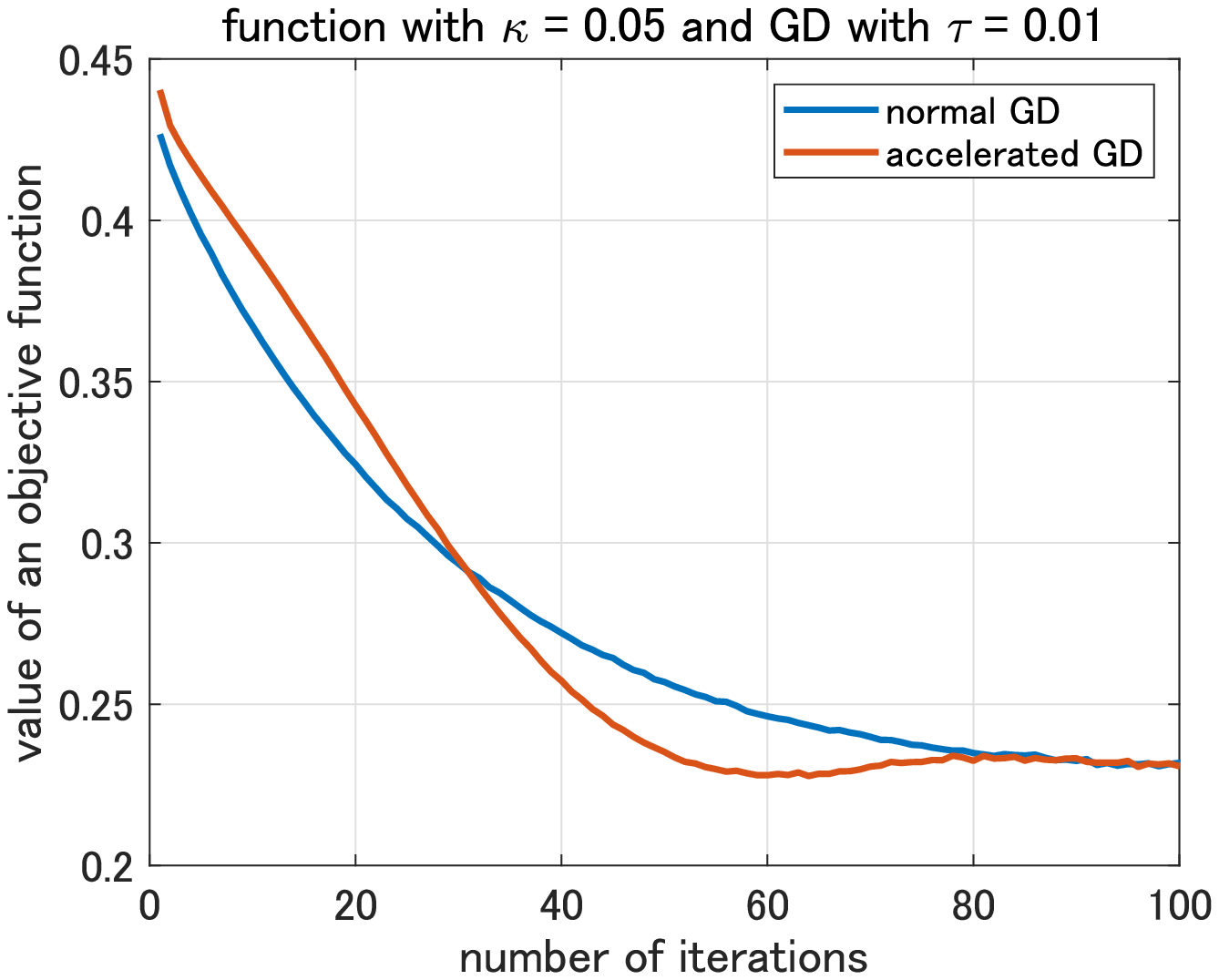}
\end{minipage} \ 
\begin{minipage}{0.3\linewidth}
\includegraphics[width=\linewidth]{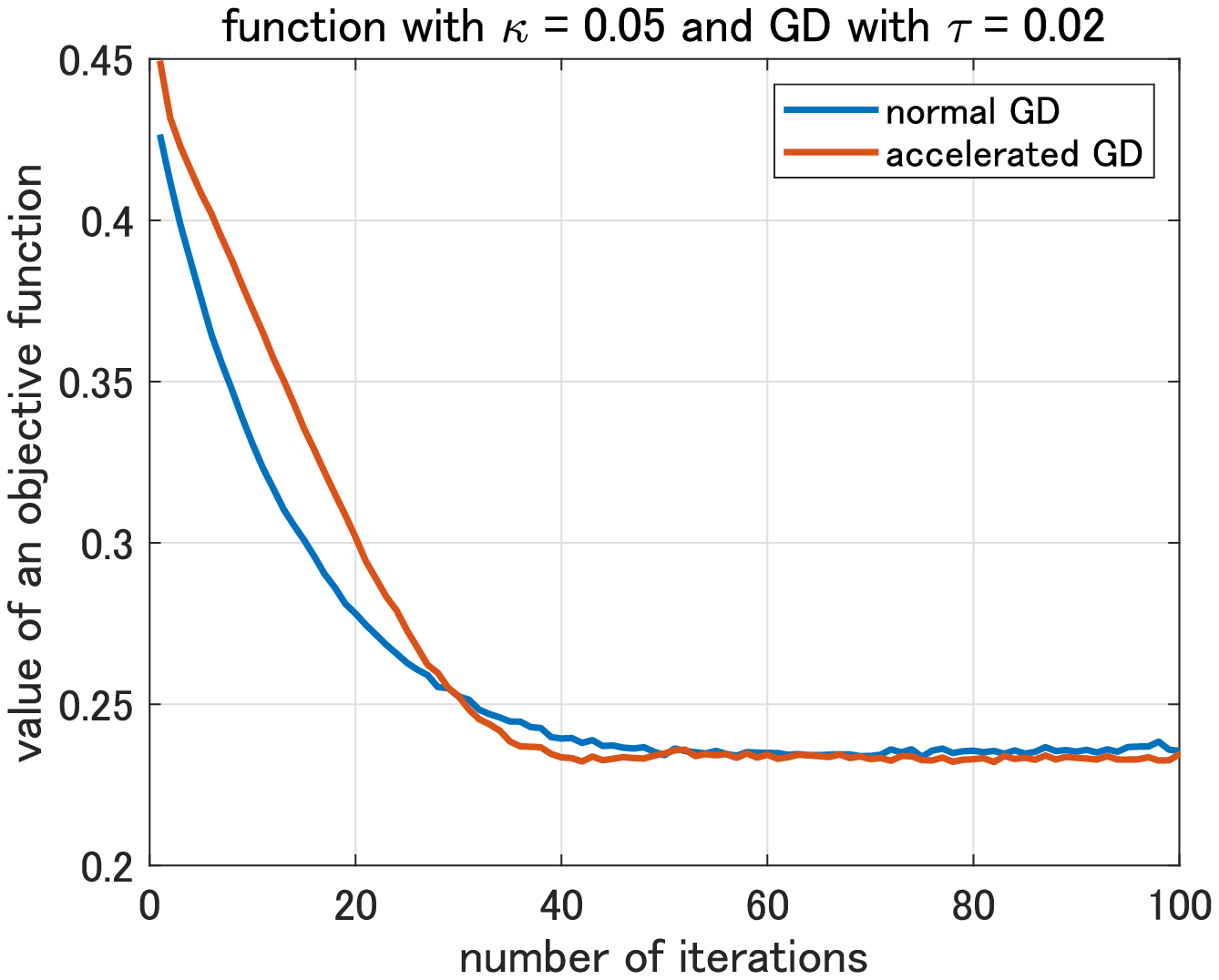}
\end{minipage}

\smallskip

\begin{minipage}{0.3\linewidth}
\includegraphics[width=\linewidth]{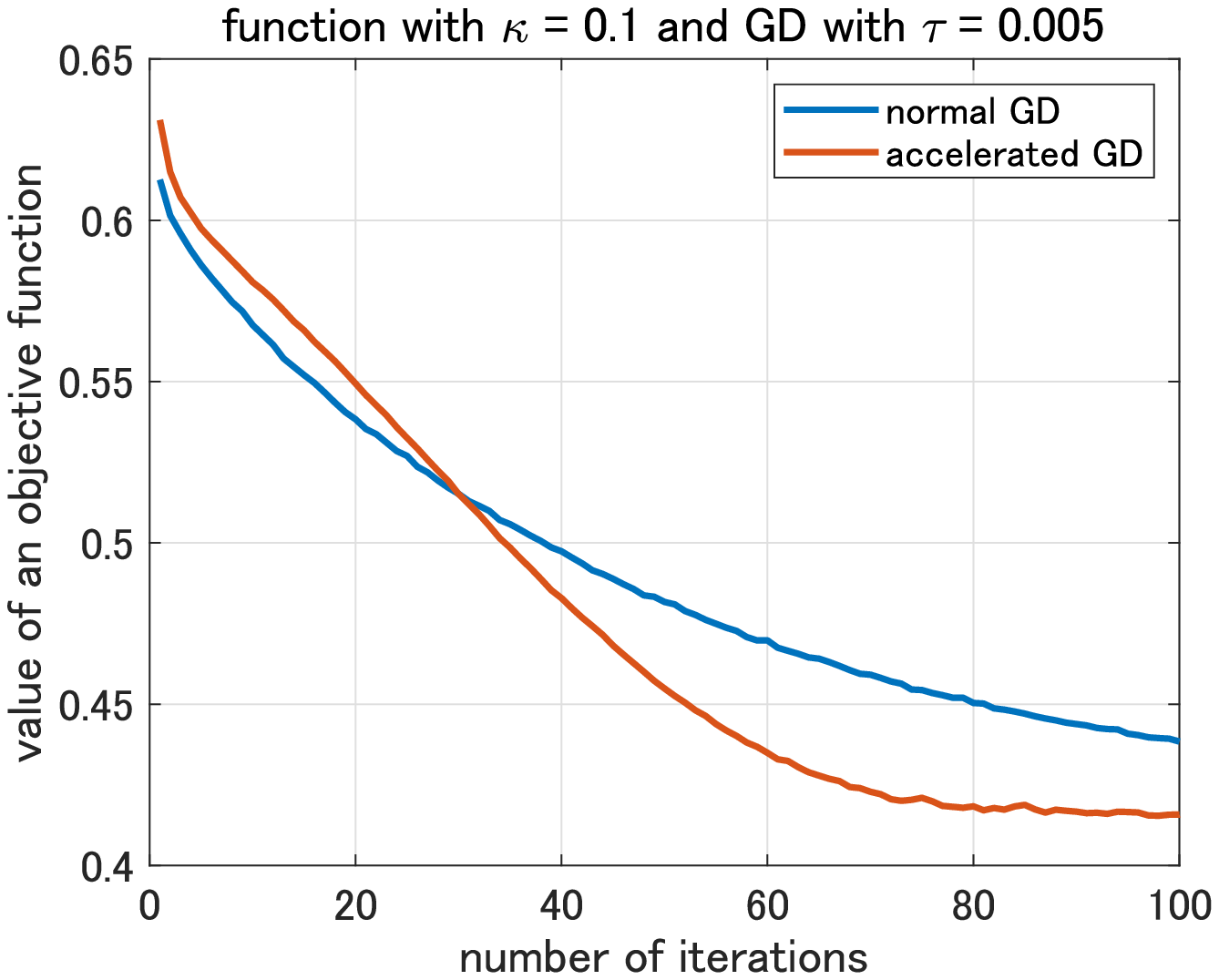}
\end{minipage} \ 
\begin{minipage}{0.3\linewidth}
\includegraphics[width=\linewidth]{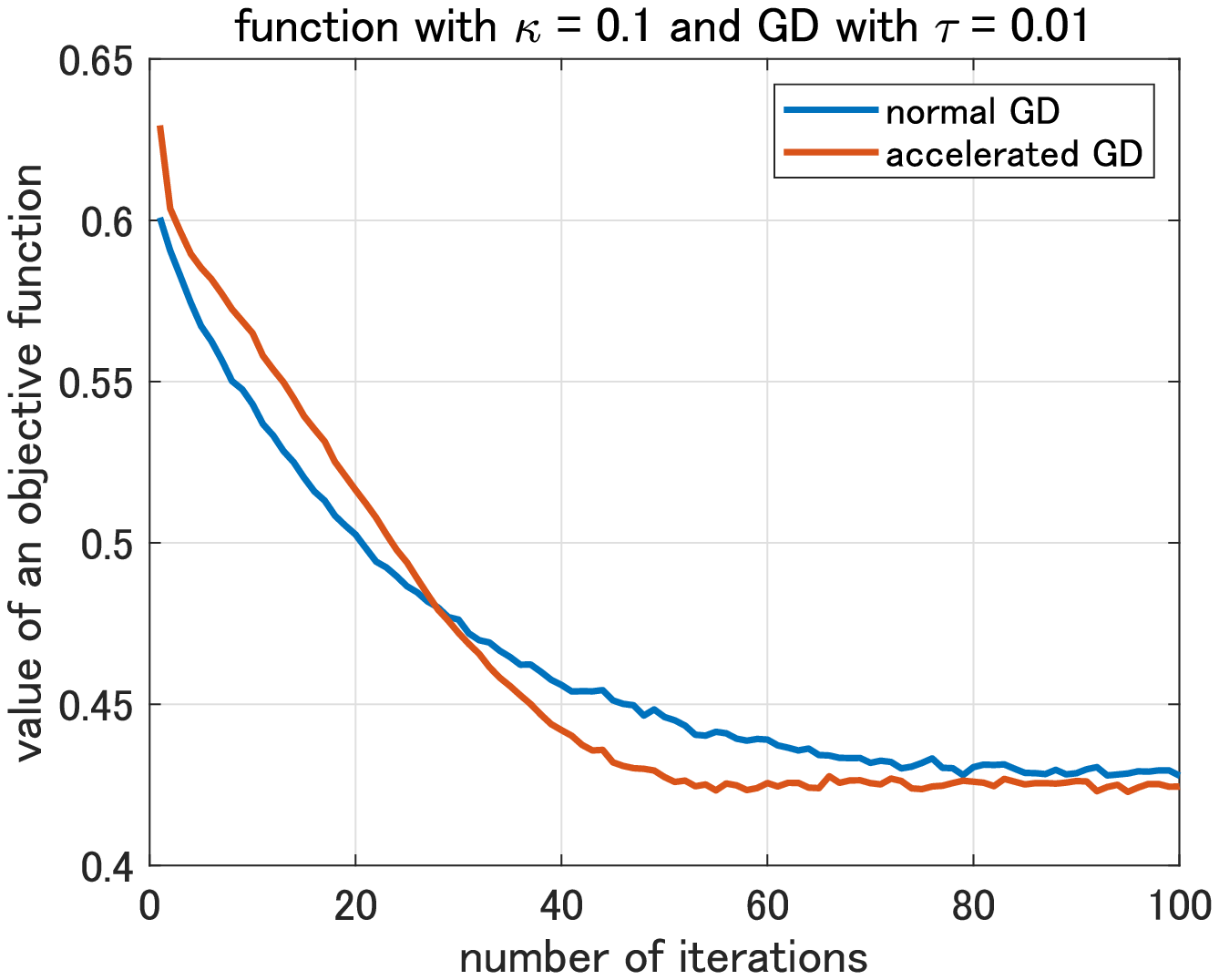}
\end{minipage} \ 
\begin{minipage}{0.3\linewidth}
\includegraphics[width=\linewidth]{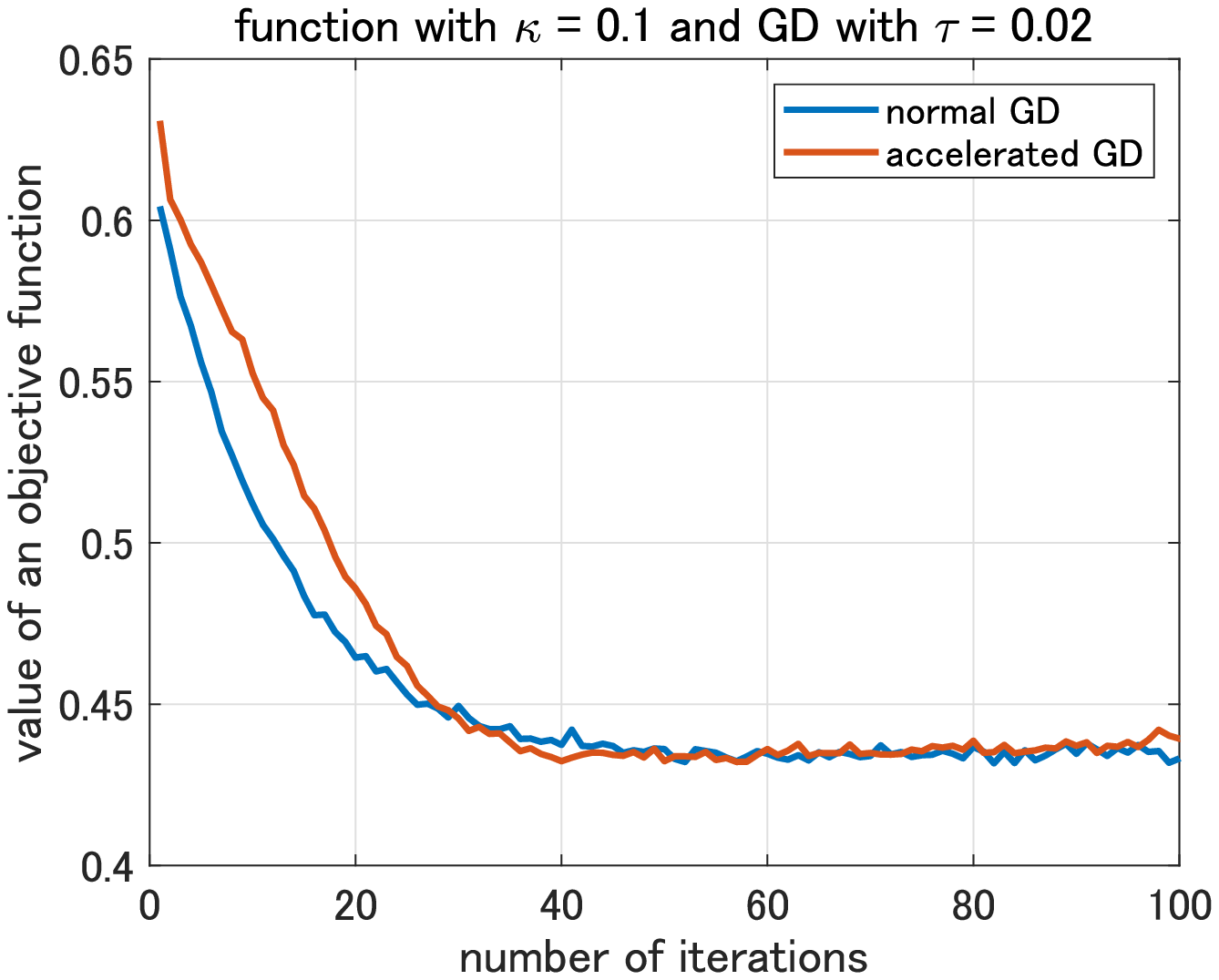}
\end{minipage}
\caption{Results for  $\mathcal{F}_{2}$ in \eqref{eq:ex_for_num_comp_2} ($2$-D negative entropy plus a linear function). }
\label{fig:two_dim_neg_ent}
\end{figure}

\section{Concluding remarks and topics for future work}
\label{sec:concl}

The gradient descent algorithm and its acceleration for functionals of probability measures 
are based on transport maps that push an initial measure forward to an optimal one. 
Those algorithms are derived from the discretization of the gradient flow equation and 
its accelerated version,
which is realized by inheriting the momentum technique 
in Nesterov's method for optimization in the Euclidean spaces. 
This inheritance is easily enabled 
because we consider the acceleration based on transport maps 
without using the Wasserstein structure of the measure space $\mathcal{P}_{2}(\mathbf{R}^{d})$. 

There are several topics for future work. 
First, 
the theory for the functional $\mathcal{U}$ is far from being complete. 
We need such a theory for explaining the result of the numerical experiments in Section~\ref{sec:num_ex_2_neg_ent}.
Next, 
we should consider ``strong convexity'' for functionals of probability measures. 
Typical functionals will have that convexity and it will contribute to improving the algorithms. 
Next, 
extending the proposed framework to a wider class of functionals is an interesting topic. 
Functionals of repulsive interaction energy are worthy to be considered.
Finally, 
we have not analyzed the fully (time-space) discretized versions of the flows considered in Sections~\ref{sec:grad_flow}. 
Such analysis is necessary for direct explanation of the real performance of the proposed algorithms. 

%
%
%
%
%
%
%

\section*{Acknowledgement}

The author thanks Prof.~Wuchen Li for several references on acceleration methods.
This work was supported by JST, PRESTO Grant Number JPMJPR2023, Japan. 

\bibliographystyle{plain}
\bibliography{opt_func_meas_grad}

\appendix
\section{Proofs}

\subsection{Proofs about gradients}
\label{sec:proofs_gradient}
\begin{proof}[Proof of Proposition~\ref{thm:grad_U}]
Let 
$\mu_{0} = \rho_{0} \mathcal{L}^{d}$
be an absolutely continuous measure with a compact support and 
let
$X(t,x)$
be a flow map whose Jacobian
$\nabla_{x} X(t,x)$
has a positive determinant for any $t$ and $x$. 
Because the density $\rho_{t}$ of the measure $\mu_{t} = X(t, \cdot)_{\#} \mu_{0}$ satisfies
\begin{align}
\rho_{t}(X(t,x)) = \frac{\rho_{0}(x)}{|\nabla_{x} X(t,x)|}, 
\notag
\end{align}
the value $\mathcal{U}(\mu_{t})$ is given by
\begin{align}
\mathcal{U}(\mu_{t})
=
\int U(\rho_{t}(y)) \, \mathrm{d}y
=
\int U\left( \frac{\rho_{0}(x)}{|\nabla_{x} X(t,x)|} \right) |\nabla_{x} X(t,x)| \, \mathrm{d}x. 
\label{eq:expr_func_U}
\end{align}
By differentiating this expression%
\footnote{
For differentiation of determinants, see \cite[Thm.~8.1]{magnus2019matrix}.}, 
we have
\begin{align}
\frac{\mathrm{d}}{\mathrm{d}t} \mathcal{U}(\mu_{t})
& = 
\int 
\left[
- U'\left( \frac{\rho_{0}(x)}{|\nabla_{x} X(t,x)|} \right) \frac{\rho_{0}(x)}{|\nabla_{x} X(t,x)|^{2}} |\nabla_{x} X(t,x)| 
+ U\left( \frac{\rho_{0}(x)}{|\nabla_{x} X(t,x)|} \right) 
\right]
\notag \\
& 
\phantom{= \int} \
\cdot |\nabla_{x} X(t,x)| \mathop{\mathrm{tr}} \left( \nabla_{x}X(t,x)^{-1} \, \frac{\partial}{\partial t} \nabla_{x} X(t,x) \right)
\, \mathrm{d}x
\notag \\
& = 
\int 
\mathop{\mathrm{tr}} \left( G[\rho_{0}, X(t,\cdot)](x) \, \nabla_{x}X(t,x)^{-1} \, \frac{\partial}{\partial t} \nabla_{x} X(t,x) \right)
\, \mathrm{d}x, 
\label{eq:diff_U_mu_t}
\end{align}
where $G[\cdot, \cdot](\cdot)$ is the scalar function defined by
\begin{align}
G[\eta, Q](x) 
:= 
- U'\left( \frac{\eta(x)}{|\nabla_{x} Q(x)|} \right) \eta(x) 
+ U\left( \frac{\eta(x)}{|\nabla_{x} Q(x)|} \right) |\nabla_{x} Q(x)|
\label{eq:def_G_eta_Q}
\end{align}
for a density function $\eta$ and a map $Q: \mathbf{R}^{d} \to \mathbf{R}^{d}$. 
By integration by parts, 
we deduce from \eqref{eq:diff_U_mu_t} that
\begin{align}
\frac{\mathrm{d}}{\mathrm{d}t} \mathcal{U}(\mu_{t})
& =
- \int 
\left \langle \mathrm{div}_{x} \left( G[\rho_{0}, X(t,\cdot)](x) \, \nabla_{x}X(t,x)^{-1} \right), \, \frac{\partial}{\partial t} X(t,x) \right \rangle
\, \mathrm{d}x
\notag \\
& = 
\int 
\left \langle - \frac{1}{\rho_{0}(x)} \, \mathrm{div}_{x} \left( G[\rho_{0}, X(t,\cdot)](x) \, \nabla_{x}X(t,x)^{-1} \right), \, \frac{\partial}{\partial t} X(t,x) \right \rangle
\, \mathrm{d}\mu_{0}(x),
\label{eq:func_U_int_by_parts}
\end{align}
where the first equality is based on the assumptions for $U$
and 
the fact that the support of $\rho_{0}$ is compact.
Therefore we obtain the gradient in~\eqref{eq:def_grad_for_U_internal}
by replacing $\mu_{0}$, $\rho_{0}$, and $X(t, \cdot)$ in
\[
- \frac{1}{\rho_{0}(x)} \, \mathrm{div}_{x} \left( G[\rho_{0}, X(t,\cdot)](x) \, \nabla_{x}X(t,x)^{-1} \right)
\]
with $\mu$, $\rho$, and $F$, respectively. 
\end{proof}

\begin{proof}[Proof of Proposition~\ref{thm:grad_V}]
The assertion follows from
\begin{align}
\frac{\mathrm{d}}{\mathrm{d}t}\mathcal{V}(\mu_{t})
& = 
\int  
\left \langle \nabla V(X(t,x)), \frac{\partial}{\partial t} X(t,x) \right \rangle
\mathrm{d}\mu_{0}(x).
\notag
\end{align}
\end{proof}

\begin{proof}[Proof of Proposition~\ref{thm:grad_W}]
The derivative of $W(X(t,x), X(t,y))$ with respect to $t$ is given by
\begin{align}
& \frac{\partial}{\partial t}
W(X(t,x), X(t,y))
\notag \\
& =  
\left \langle \nabla_{1} W(X(t,x), X(t,y)), \frac{\partial}{\partial t} X(t,x) \right \rangle
+
\left \langle \nabla_{2} W(X(t,x), X(t,y)), \frac{\partial}{\partial t} X(t,y) \right \rangle
\notag \\
& = 
\left \langle \nabla_{1} W(X(t,x), X(t,y)), \frac{\partial}{\partial t} X(t,x) \right \rangle
+
\left \langle \nabla_{1} W(X(t,y), X(t,x)), \frac{\partial}{\partial t} X(t,y) \right \rangle, 
\label{eq:diff_F_mu_t}
\end{align}
where the second equality is owing to the relation $\nabla_{1}W(x,y) = \nabla_{2}W(y,x)$, 
that is derived from the symmetry of $W$. 
By integrating both sides of \eqref{eq:diff_F_mu_t} 
by $\mathrm{d}\mu_{0}(x)$ and $\mathrm{d}\mu_{0}(y)$, 
we have
\begin{align}
\frac{\mathrm{d}}{\mathrm{d}t} \mathcal{F}(\mu_{t})
& =
\int 
\left \langle \int \nabla_{1} W(X(t,x), X(t,y)) \, \mathrm{d}\mu_{0}(y), \frac{\partial}{\partial t} X(t,x) \right \rangle
\mathrm{d}\mu_{0}(x)
\notag \\
& =
\int 
\left \langle \int \nabla_{1} W(X(t,x), y) \, \mathrm{d}\mu_{t}(y), \frac{\partial}{\partial t} X(t,x) \right \rangle
\mathrm{d}\mu_{0}(x), 
\notag
\end{align}
which implies the assertion. 
\end{proof}

\subsection{Proofs about convexity and smoothness}
\label{sec:proofs_convex_smooth}
\subsubsection{Proofs about convexity}

\begin{proof}[Proof of Theorem~\ref{thm:conv_of_U}]
Let $\mu$ be an absolutely continuous probability measure whose support is compact.
Let $\rho$ be the density of $\mu$. 
Furthermore, 
let $S$ and $T$ be maps from $\mathbf{R}^{d}$ to $\mathbf{R}^{d}$ 
whose Jacobians $\nabla S(x)$ and $\nabla T(x)$ are positive definite for any $x$. 
For $t \in [0,1]$, we define $R(t, x)$ by
\begin{align}
R(t, x) := (1-t) S(x) + t \, T(x).
\label{eq:def_R_by_S_and_T}
\end{align}

If we show that 
for any $x$ with $\rho(x) > 0$
the function defined by
\begin{align}
r(t,x) := U\left( \frac{\rho(x)}{|\nabla_{x} R(t,x)|} \right) |\nabla_{x} R(t,x)|
\label{eq:def_func_r_by_U_rho_R}
\end{align}
is convex as a function of $t \in [0,1]$,  
we have the inequality
\begin{align}
r(1,x) - r(0,x) \geq r_{t}(0,x), 
\notag 
\end{align}
where $r_{t}$ denotes the partial derivative of $r$ with respect to $t$.
By using the same manner as the calculation in \eqref{eq:diff_U_mu_t} for $r_{t}(0,x)$, 
we deduce from the above inequality that 
\begin{align}
& U\left( \frac{\rho(x)}{|\nabla T(x)|} \right) |\nabla T(x)|
- 
U\left( \frac{\rho(x)}{|\nabla S(x)|} \right) |\nabla S(x)|
\notag \\
& \geq
\mathop{\mathrm{tr}} \left( G[\rho, R(0, \cdot)](x)  \, \nabla_{x}R(0,x)^{-1} \, \frac{\partial}{\partial t} \nabla_{x} R(t,x) \Big|_{t=0} \right)
\notag \\
& =
\mathop{\mathrm{tr}} \left( G[\rho, S](x) \, \nabla S(x)^{-1} \, \nabla (T(x) - S(x)) \right), 
\label{eq:func_U_geq_grad_by_trace}
\end{align}
where $G[\cdot, \cdot](\cdot)$ is defined by \eqref{eq:def_G_eta_Q}. 
By integrating both sides of the above inequality with respect to $\mathrm{d}x$, 
in the same manner as \eqref{eq:func_U_int_by_parts}
we have
\begin{align}
\mathcal{U}(T_{\#} \mu) - \mathcal{U}(S_{\#} \mu)
& \geq 
\int 
\left \langle - \frac{1}{\rho(x)} \, \mathrm{div} \left( G[\rho, S](x) \, \nabla S(x)^{-1} \right), \, T(x) - S(x) \right \rangle
\, \mathrm{d}\mu(x)
\notag \\
& = 
\int 
\left \langle \mathcal{G}_{\mathcal{U}}[S_{\#}\mu](S(x)), \, T(x) - S(x) \right \rangle
\, \mathrm{d}\mu(x), 
\label{eq:func_U_convex}
\end{align}
which is the conclusion. 

In the following, 
we prove the convexity of the function $r(\cdot, x)$ given by \eqref{eq:def_func_r_by_U_rho_R}. 
To this end, 
we regard $r(\cdot, x)$ as the composition $\alpha \circ \beta$ of the functions $\alpha$ and $\beta$ defined by
\begin{align}
& \alpha(s) := U\left( \rho(x) \, s^{-d} \right) \, s^{d}
\qquad
\left( = \tilde{U}_{d}\left( s/\rho(x)^{1/d} \right) \rho(x) \right)
\qquad \text{and}
\notag \\
& \beta(t) := |\nabla_{x} R(t,x)|^{1/d}
\qquad 
\left( = \left| (1-t) \nabla S(x) + t\, \nabla T(x) \right|^{1/d} \right), 
\notag
\end{align}
respectively. 
Because of the assumption about $\tilde{U}_{d}$ in \eqref{eq:def_tilde_U}, 
the function $\alpha$ is convex and non-increasing. 
Moreover, 
the function $\beta$ is concave. 
To show this, 
let $t_{0}$ and $t_{1}$ be real numbers in $[0,1]$ with $t_{0} < t_{1}$ and 
let $t_{\lambda} := (1 - \lambda) t_{0} + \lambda t_{1}$ for $\lambda \in (0,1)$. 
Then, we can show the concavity of $\beta$ as follows:
\begin{align}
\beta(t_{\lambda})
& =
\left| \nabla_{x} R(t_{\lambda},x) \right|^{1/d}
\notag \\
& = 
\left| (1-\lambda) \nabla_{x} R(t_{0},x) + \lambda \nabla_{x} R(t_{1},x) \right|^{1/d}
\notag \\
& \geq
\left| (1-\lambda) \nabla_{x} R(t_{0},x) \right|^{1/d}
+
\left| \lambda \nabla_{x} R(t_{1},x) \right|^{1/d}
\notag \\
& = 
(1-\lambda) \left| \nabla_{x} R(t_{0},x) \right|^{1/d}
+
\lambda \, \left| \nabla_{x} R(t_{1},x) \right|^{1/d}
\notag \\
& = 
(1-\lambda) \beta(t_{0}) + \lambda \beta(t_{1}), 
\notag
\end{align}
where the inequality is owing to the Minkowski inequality for matrices%
\footnote{For $d \times d$ positive semi-definite matrices $A$ and $B$, 
\[
|A + B|^{1/d} \geq |A|^{1/d} + |B|^{1/d}
\]
holds. 
See \cite[Lem.~15.15]{ambrosio2021lectures}, \cite[Exercise 12.13]{abadir_magnus_2005}, or \cite[Thm.~11.28]{magnus2019matrix}.}.
Therefore $r(\cdot, x) = \alpha \circ \beta$ is convex because
\begin{align}
\alpha(\beta(t_{\lambda}))
\leq
\alpha((1-\lambda) \beta(t_{0}) + \lambda \beta(t_{1}))
\leq
(1-\lambda) \, \alpha(\beta(t_{0}))
+ 
\lambda \, \alpha(\beta(t_{1})).
\notag
\end{align}
Thus Theorem~\ref{thm:conv_of_U} is proven. 
\end{proof}

\begin{proof}[Proof of Theorem~\ref{thm:conv_of_V}]
By the convexity of $V$, we have
\begin{align}
V(T(x)) - V(S(x))
\geq
\left \langle
\nabla V(S(x)), \, T(x) - S(x) 
\right \rangle.
\notag
\end{align}
By integrating both sides of the above inequality with respect to $\mathrm{d}\mu(x)$, 
we can show the {\newconvexity} of $\mathcal{V}$. 
\end{proof}

\begin{proof}[Proof of Theorem~\ref{thm:conv_of_W}]
Let $S, T \in L^{2}(\mathbf{R}^{d} \to \mathbf{R}^{d})$ and $\mu \in \mathcal{P}_{2}(\mathbf{R}^{d})$. 
Then, 
by the convexity of $\tilde{W}$, we have
\begin{align}
& \tilde{W}(T(x) - T(y)) - \tilde{W}(S(x) - S(y)) 
\notag \\
& \geq
\left \langle
\nabla \tilde{W}(S(x) - S(y)), \, T(x) - T(y) - (S(x) - S(y))
\right \rangle
\notag \\
& = 
\left \langle
\nabla \tilde{W}(S(x) - S(y)), \, T(x) - S(x)
\right \rangle
+
\left \langle
- \nabla \tilde{W}(S(x) - S(y)), \, T(y) - S(y)
\right \rangle
\notag \\
& = 
\left \langle
\nabla \tilde{W}(S(x) - S(y)), \, T(x) - S(x)
\right \rangle
+
\left \langle
\nabla \tilde{W}(S(y) - S(x)), \, T(y) - S(y)
\right \rangle, 
\notag
\end{align}
where the last equality is owing to $-\nabla \tilde{W}(x) = \nabla \tilde{W}(-x)$. 
By integrating both sides of the above inequality with respect to $\mathrm{d}\mu(x)$ and $\mathrm{d}\mu(y)$, 
we have
\begin{align}
2 \mathcal{W}(T_{\#} \mu) - 2 \mathcal{W}(S_{\#} \mu) 
& \geq 
2
\int  
\left \langle
\int  \nabla \tilde{W}(S(x) - S(y)) \, \mathrm{d}\mu(y), \, T(x) - S(x)
\right \rangle
\mathrm{d}\mu(x)
\notag \\
& = 
2
\int 
\left \langle
\mathcal{G}_{\mathcal{W}}[S_{\#} \mu](S(x)), \, T(x) - S(x)
\right \rangle
\mathrm{d}\mu(x). 
\notag 
\end{align}
This inequality indicates the {\newconvexity} of $\mathcal{W}$. 
\end{proof}

\subsubsection{Proofs about smoothness}
\label{sec:proofs_smooth}

We consider the following assumptions for 
Theorem~\ref{thm:func_U_restricted_L_smooth}
about the functional $\mathcal{U}$. 
In these assumptions, 
$\rho$ is the density of $\mu \in \mathcal{P}_{2}(\mathbf{R}^{d})$, 
$U: [0, \infty) \to (-\infty, \infty]$
is the function with~\eqref{eq:diff_U_assump} and~\eqref{eq:diff_U_assump_2nd} defining $\mathcal{U}$, 
and
$S$ and $T$ are maps from $\mathbf{R}^{d}$ to $\mathbf{R}^{d}$ 
whose Jacobians $\nabla S(x)$ and $\nabla T(x)$ are positive definite for any $x$. 
Recall that the function $G[\cdot, \cdot](\cdot)$ is defined by~\eqref{eq:def_G_eta_Q}. 
In addition, 
for the map $R(t,x)$ defined by~\eqref{eq:def_R_by_S_and_T}, 
let $A(t,x)$ be the matrix defined by
\begin{align}
A(t,x) := \nabla_{x}R(t,x)^{-1} ( \nabla T(x) - \nabla S(x) ).
\label{eq:def_mat_A}
\end{align}
for $t \in [0,1]$.

\begin{assump}
\label{assump:bounded_C1_C2}
There exist real constants $C_{1}$ and $C_{2}$ such that
\begin{align}
& \left| 
\frac{G[\rho, R(t, \cdot)](x)}{\rho(x)}
\right| 
\leq C_{1} \quad \text{and} 
\notag \\
& \left|
U''\left(
\frac{\rho(x)}{|\nabla_{x} R(t,x)|}
\right)
\frac{\rho(x)}{|\nabla_{x} R(t,x)|}
\right|
\leq C_{2}
\notag
\end{align}
hold for any $t \in [0,1]$ and $x \in \mathop{\mathrm{supp}} \rho$. 
\end{assump}

\begin{assump}
\label{assump:matrix_A_PD}
The matrix $A(t,x)$ in~\eqref{eq:def_mat_A} is positive definite for any $t \in [0,1]$ and $x \in \mathop{\mathrm{supp}} \rho$.
\end{assump}

\begin{assump}
\label{assump:bounded_C3_L2_L1}
Let $C_{3}$ be a positive real constant. 
The maps $S$ and $T$ belong to the class of maps such that 
the matrix $A(t,x)$ defined by~\eqref{eq:def_mat_A} with these maps satisfies 
\begin{align}
\int_{\mathop{\mathrm{supp}} \rho} 
\mathop{\mathrm{tr}} 
\left(
A(t,x)
\right)^{2}
\mathrm{d}x
\leq
C_{3}
\left( 
\int_{\mathop{\mathrm{supp}} \rho}
\mathop{\mathrm{tr}} 
\left(
A(t,x)
\right)
\mathrm{d}x
\right)^{2}.
\notag
\end{align}
\end{assump}

\begin{assump}
\label{assump:bounded_C4_div}
Let $C_{4}$ be a positive real constant. 
The maps $S$ and $T$ belongs to the class of maps such that 
$\mathop{\mathrm{supp}} S$ and $\mathop{\mathrm{supp}} T$ are contained in a set 
$\mathscr{S} \subset \mathop{\mathrm{supp}} \rho$
and
the matrix $R(t,x)$ defined by~\eqref{eq:def_R_by_S_and_T} with these maps satisfies 
\begin{align}
\int_{\mathscr{S}}
\frac{1}{\rho(x)}
\left\| \mathop{\mathrm{div}} \left( \sqrt{\rho(x)} \nabla_{x} R(t,x)^{-1} \right) \right\|^{2} 
\mathrm{d}x
\leq C_{4}.
\notag
\end{align}
\end{assump}

\begin{rem}
\label{rem:L_smooth_U_assump}
Assumptions~\ref{assump:bounded_C1_C2}--\ref{assump:bounded_C4_div}
restrict the class of the maps $S$ and $T$ in a complicated manner. 
Therefore Theorem~\ref{thm:func_U_restricted_L_smooth} under these assumptions cannot be used to prove
Theorems~\ref{thm:c_rate_disc_flow} and~\ref{thm:c_rate_acc_disc_flow} for the functional $\mathcal{F} = \mathcal{U}$. 
These proofs are left to future work. 

Below we roughly observe what $\rho$, $U$, $S$, and $T$ satisfy the assumptions.
First, 
we consider Assumption~\ref{assump:bounded_C1_C2} 
for the typical example $U(x) = x \log x$ treated in Remarks~\ref{rem:entropy} and~\ref{rem:KL_div}. 
Because $U'(x) = \log x + 1$ and $U''(x) = 1/x$, we have
\begin{align}
\frac{G[\rho, R(t, \cdot)](x)}{\rho(x)} = -1 
\quad \text{and} \quad
U''\left(
\frac{\rho(x)}{|\nabla_{x} R(t,x)|}
\right)
\frac{\rho(x)}{|\nabla_{x} R(t,x)|} = 1.
\notag
\end{align}
Therefore Assumption~\ref{assump:bounded_C1_C2} is satisfied. 
Next, 
it is clear that Assumption~\ref{assump:matrix_A_PD} is satisfied if $\nabla T(x) - \nabla S(x) \succ O$. 
Next, 
to consider Assumption~\ref{assump:bounded_C3_L2_L1}, 
we assume that 
$\nabla S(x) \approx J_{1}$ 
and 
$\nabla T(x) - \nabla S(x) \approx J_{2}$
for constant matrices $J_{1}$ and $J_{2}$ with $J_{2} \approx O$.  
Then, 
$\nabla_{x} R(t,x) \approx J_{1}$ and $A(t, x) \approx J_{1}^{-1} J_{2}$ hold. 
In addition, 
we assume that $\mathop{\mathrm{supp}} \rho$ is compact. 
Then, 
the inequality of Assumption~\ref{assump:bounded_C3_L2_L1} holds for 
$C_{3} \approx 1/\mathcal{L}^{d}(\mathop{\mathrm{supp}} \rho)$
because $\mathop{\mathrm{tr}} (A(t,x))$ is close to the constant $\mathop{\mathrm{tr}} (J_{1}^{-1} J_{2})$.
Finally, 
for Assumption~\ref{assump:bounded_C4_div}, 
we consider the same situation as is adopted for Assumption~\ref{assump:bounded_C3_L2_L1}. 
In addition, 
we assume that $\rho(x)$ is $\mathrm{C}^{1}$ and larger than a positive constant on $\mathscr{S}$. 
Then, 
Assumption~\ref{assump:bounded_C4_div} is satisfied. 
\end{rem}

\begin{proof}[Proof of Theorem~\ref{thm:func_U_restricted_L_smooth}]
By applying the Taylor theorem to the function $r(\cdot, x)$ given by \eqref{eq:def_func_r_by_U_rho_R}, 
we get 
\[
r(1,x) = r(0,x) + r_{t}(0,x) + \frac{r_{tt}(t_{\ast}, x)}{2}
\] 
for some $t_{\ast} \in [0,1]$, 
where $r_{tt}$ denotes the second-order partial derivative of $r$ with respect to $t$.
Recall that 
integration of $r(1,x)$, $r(0,x)$, and $r_{t}(0,x)$ with respect to $\mathrm{d}x$ provides 
$\mathcal{U}(T_{\#} \mu)$, 
$\mathcal{U}(S_{\#} \mu)$, and
the RHS of \eqref{eq:func_U_convex}, respectively. 
Therefore
the {\newsmoothness} of $\mathcal{U}$ in Section~\ref{sec:func_internal_ene} is proved
if  
\begin{align}
\sup_{t \in [0,1]} \left| \int r_{tt}(t,x) \, \mathrm{d}x \right| \leq L \int \| T(x) - S(x) \|^{2} \, \rho(x) \, \mathrm{d}x
\label{eq:func_U_deg2_term_ub}
\end{align}
holds for $L \geq 0$.

To show inequality~\eqref{eq:func_U_deg2_term_ub}, 
we derive the expression of $r_{tt}(t,x)$. 
By replacing $X(t,x)$ in~\eqref{eq:diff_U_mu_t} with $R(t,x)$ in~\eqref{eq:def_R_by_S_and_T}, 
we have
\begin{align}
r_{t}(t,x)
= 
\mathop{\mathrm{tr}} \left( G[\rho, R(t,\cdot)](x) \, \nabla_{x}R(t,x)^{-1} \, (\nabla T(x) - \nabla S(x)) \right),
\notag
\end{align}
where $G[\cdot, \cdot](\cdot)$ is defined by \eqref{eq:def_G_eta_Q}. 
Therefore $r_{tt}(t,x)$ is given by
\begin{align}
r_{tt}(t,x)
& = 
\mathop{\mathrm{tr}}
\left[
\left(
U''\left(
\frac{\rho(x)}{|\nabla_{x} R(t,x)|}
\right)
\frac{\rho(x)^{2}}{|\nabla_{x} R(t,x)|}
+ 
G[\rho, R(t,\cdot)](x)
\right)
\cdot 
\mathop{\mathrm{tr}} 
\left(
A(t,x)
\right)
\cdot 
A(t,x)
\right.
\notag \\
& \phantom{= \mathop{\mathrm{tr}} \bigg[}
- G[\rho, R(t,\cdot)](x) \cdot A(t,x)^{2} \bigg]
\notag \\
& 
= 
\left(
U''\left(
\frac{\rho(x)}{|\nabla_{x} R(t,x)|}
\right)
\frac{\rho(x)^{2}}{|\nabla_{x} R(t,x)|}
+ 
G[\rho, R(t, \cdot)](x)
\right)
\mathop{\mathrm{tr}} 
\left(
A(t,x)
\right)^{2}
\notag \\
& 
\phantom{=} \  
- 
G[\rho, R(t, \cdot)](x) \, 
\mathop{\mathrm{tr}} 
\left(
A(t,x)^{2}
\right), 
\notag
\end{align}
where $A(t,x)$ is defined by~\eqref{eq:def_mat_A}.
Here we emphasize that $r_{tt}(t, x) = 0$ for $x \not \in \mathop{\mathrm{supp}} \rho$
because
\begin{align}
U''\left(
\frac{\rho(x)}{|\nabla_{x} R(t,x)|}
\right)
\frac{\rho(x)^{2}}{|\nabla_{x} R(t,x)|} = 0
\quad \text{and} \quad 
G[\rho, R(t, \cdot)](x) = 0
\end{align}
for $x \not \in \mathop{\mathrm{supp}} \rho$ owing to~\eqref{eq:diff_U_assump}. 
Therefore 
the domain of the integrals in~\eqref{eq:func_U_deg2_term_ub} 
can be restricted to $\mathop{\mathrm{supp}} \rho$. 
Then we assume that $x \in \mathop{\mathrm{supp}} \rho$ in the following. 

By Assumption~\ref{assump:bounded_C1_C2},
the value $\left| r_{tt}(t,x) \right|$ is bounded as follows: 
\begin{align}
\left| r_{tt}(t,x) \right|
\leq
(C_{2} + 2 C_{1}) \, 
\rho(x) 
\mathop{\mathrm{tr}} 
\left(
A(t,x)
\right)^{2},
\label{eq:abs_rp0_ub}
\end{align}
where we use 
\(
\mathop{\mathrm{tr}} 
\left(
A(t,x)^{2}
\right)
\leq 
\mathop{\mathrm{tr}} 
\left(
A(t,x)
\right)^{2}
\)
because $A(t,x) \succ O$ is guaranteed by 
Assumption~\ref{assump:matrix_A_PD}%
\footnote{
In general, 
the inequality $\mathop{\mathrm{tr}}(AB) \leq \mathop{\mathrm{tr}}(A) \mathop{\mathrm{tr}}(B)$ 
holds for positive semi-definite matrices $A$ and $B$. 
See \cite[Exercise 12.14]{abadir_magnus_2005}.
}.
From Assumptions~\ref{assump:bounded_C3_L2_L1} and~\ref{assump:bounded_C4_div}, 
we can derive an upper bound of the integral of 
\(
\rho(x) 
\mathop{\mathrm{tr}} 
\left(
A(t,x)
\right)^{2}
\) 
in~\eqref{eq:abs_rp0_ub} as follows:
\begin{align}
& \int_{\mathop{\mathrm{supp}} \rho} 
\rho(x) 
\mathop{\mathrm{tr}} 
\left(
A(t,x)
\right)^{2}
\mathrm{d}x
\notag \\
& \leq
C_{3}
\left( 
\int_{\mathop{\mathrm{supp}} \rho}
\mathop{\mathrm{tr}} 
\left(
\sqrt{\rho(x)}
A(t,x)
\right)
\mathrm{d}x
\right)^{2}
\notag \\
& = 
C_{3}
\left( 
\int_{\mathscr{S}}
\mathop{\mathrm{tr}} 
\left(
\sqrt{\rho(x)}
\nabla_{x} R(t,x)^{-1} ( \nabla T(x) - \nabla S(x) )
\right)
\mathrm{d}x
\right)^{2}
\notag \\
& = 
C_{3}
\left( 
\int_{\mathscr{S}}
\left \langle
\mathop{\mathrm{div}} 
\left(
\sqrt{\rho(x)} \nabla_{x} R(t,x)^{-1} 
\right), \ 
T(x) - S(x)
\right \rangle
\mathrm{d}x
\right)^{2}
\notag \\
& \leq 
C_{3}
\int_{\mathscr{S}}
\frac{1}{\rho(x)}
\left\| \mathop{\mathrm{div}} \left( \sqrt{\rho(x)} \nabla_{x} R(t,x)^{-1} \right) \right\|^{2} 
\mathrm{d}x
\int_{\mathscr{S}}
\left\| T(x) - S(x) \right \|^{2} \rho(x) \, \mathrm{d}x
\notag \\
& \leq 
C_{3} C_{4}
\int_{\mathscr{S}}
\left\| T(x) - S(x) \right \|^{2} \rho(x) \, \mathrm{d}x.
\label{eq:abs_rp0_int_ub}
\end{align}
Therefore 
it follows from \eqref{eq:abs_rp0_ub} and \eqref{eq:abs_rp0_int_ub} that
\begin{align}
\left| \int_{\mathop{\mathrm{supp}} \rho} r_{tt}(t,x) \, \mathrm{d}x \right|
\leq 
 (2 C_{1} + C_{2}) C_{3} C_{4}
\int_{\mathscr{S}}
\left\| T(x) - S(x) \right \|^{2} \rho(x) \, \mathrm{d}x
\notag
\end{align}
holds for any $t \in [0,1]$. 
Thus we have inequality~\eqref{eq:func_U_deg2_term_ub} for 
$L = (2 C_{1} + C_{2}) C_{3} C_{4}$.
\end{proof}

\begin{proof}[Proof of Theorem~\ref{thm:func_pot_L_smooth}]
From \eqref{eq:V_L_smooth}, 
we have
\begin{align}
V(T(x)) - V(S(x)) 
\leq 
\left \langle \nabla V(S(x)), T(x) - S(x) \right \rangle
+ \frac{L_{V}}{2} \| T(x) - S(x) \|^{2}. 
\end{align}
By integrating both sides of the above inequality with respect to $\mathrm{d}\mu(x)$, 
we can show the {\newsmoothness} of $\mathcal{V}$ with $L = L_{V}$. 
\end{proof}

\begin{proof}[Proof of Theorem~\ref{thm:func_inter_L_smooth}]
Let $S, T \in L^{2}(\mathbf{R}^{d} \to \mathbf{R}^{d})$ and $\mu \in \mathcal{P}_{2}(\mathbf{R}^{d})$. 
It follows from the Taylor theorem that
\begin{align}
& W(T(x), T(y))
\notag \\
& = 
W(S(x) + (T(x)-S(x)), S(y) + (T(y)-S(y))
\notag \\
& = 
W(S(x), S(y)) 
\notag \\
& \phantom{=} \, 
+ 
\left \langle
\nabla_{1}W(S(x), S(y)), \, T(x) - S(x)
\right \rangle
\notag \\
& \phantom{=} \, 
+ 
\left \langle
\nabla_{2}W(S(x), S(y)), \, T(y) - S(y)
\right \rangle
+
\frac{1}{2} \, d_{S,T}(x,y)^{\top} \, M_{\theta}[S, T](x, y) \, d_{S,T}(x,y),
\label{eq:Taylor_for_W_T_T}
\end{align}
where $\theta \in (0, 1)$, 
\begin{align}
& d_{S,T}(x,y)
:= 
\begin{bmatrix}
T(x) - S(x) \\
T(y) - S(y)
\end{bmatrix}, \quad \text{and}
\notag \\ 
& M_{\theta}[S, T](x, y)
:= 
H_{W} \Big( S(x) + \theta (T(x) - S(x)), \, S(y) + \theta (T(y) - S(y)) \Big). 
\notag
\end{align}
Under the assumption about $r_{W}$ in~\eqref{eq:2diff_W_eigen_finite}, 
the last term in \eqref{eq:Taylor_for_W_T_T} is bounded as follows: 
\begin{align}
& \frac{1}{2} \, d_{S,T}(x,y)^{\top} \, M_{\theta}[S, T](x, y) \, d_{S,T}(x,y)
\notag \\
& \leq 
\frac{r_{W}}{2} \| d_{S,T}(x,y) \|^{2} 
= 
\frac{r_{W}}{2} 
\left( 
\| T(x) - S(x) \|^{2} + \| T(y) - S(y) \|^{2}
\right).
\label{eq:bound_2nd_term_of_Taylor}
\end{align}
From~\eqref{eq:Taylor_for_W_T_T} and~\eqref{eq:bound_2nd_term_of_Taylor}, 
we have
\begin{align}
2 \mathcal{W}(T_{\#} \mu) 
& \leq
2 \mathcal{W}(S_{\#} \mu) 
+ 
2 \int \left \langle \int  \nabla_{1} W(S(x), S(y)) \, \mathrm{d}\mu(y), T(x) - S(x) \right \rangle \mathrm{d}\mu(x) 
\notag \\
& \phantom{\leq} \, 
+ 
r_{W} \int \| T(x) - S(x) \|^{2} \, \mathrm{d}\mu(x)
\notag \\
& = 
2 \mathcal{W}(S_{\#} \mu) 
+ 
2 \int \left \langle \mathcal{G}_{\mathcal{W}}[S_{\#}\mu](S(x)), T(x) - S(x) \right \rangle \mathrm{d}\mu(x)
\notag \\
& \phantom{\leq} \, 
+ 
r_{W} \int \| T(x) - S(x) \|^{2} \, \mathrm{d}\mu(x),
\notag
\end{align}
which implies the conclusion. 
\end{proof}

\subsection{Proofs about convergence rate}
\label{sec:proofs_conv_rate}
\subsubsection{Proof of Theorem~\ref{thm:c_rate_cont_flow}}

\begin{proof}
As shown in~\eqref{eq:Lyap_for_C_flow_pre}, 
we define $\mathcal{L}(t)$ by
\begin{align}
\mathcal{L}(t) 
:= 
\mathcal{A}(t) + \mathcal{B}(t),
\label{eq:Lyap_for_C_flow}
\end{align}
where
\begin{align}
& \mathcal{A}(t) := t \, (\mathcal{F}(\mu_{t}) - \mathcal{F}(\mu_{\ast})), \quad \text{and}
\label{eq:Lyap_t_times_F} \\
& \mathcal{B}(t) := D(\mu_{t}, \mu_{\ast}).
\label{eq:non_neg_term_for_C_flow}
\end{align}
Note that 
$D(\mu_{t}, \mu_{\ast}) = D(X(t,\cdot)_{\#} \mu_{0}, (X_{\ast})_{\#} \mu_{0})$ and 
this is defined by~\eqref{eq:def_discr_func}. 

For the function $\mathcal{A}(t)$, 
it follows from Proposition \ref{prop:decrease_of_F_mu_t} that
\begin{align}
\frac{\mathrm{d}}{\mathrm{d}t} \mathcal{A}(t)
& = 
\mathcal{F}(\mu_{t}) - \mathcal{F}(\mu_{\ast})
- t \int \left \| \mathcal{G}_{\mathcal{F}}[\mu_{t}](X(t, x)) \right \|^{2} \, \mathrm{d}\mu_{0}(x). 
\label{eq:diff_A_1}
\end{align}
The derivative of the function $\mathcal{B}(t)$ is given by
\begin{align}
\frac{\mathrm{d}}{\mathrm{d}t} \mathcal{B}(t)
& = 
\int 
\left \langle \frac{\partial}{\partial t} X(t,x) , \, X(t,x) - X_{\ast}(x) \right\rangle 
\mathrm{d}\mu_{0}(x)
\notag \\
& = 
\int
\left \langle -\mathcal{G}_{\mathcal{F}}[X(t, \cdot)_{\#} \mu_{0}](X(t,x)) , \strut \,  X(t,x) - X_{\ast}(x) \right \rangle 
\mathrm{d}\mu_{0}(x)
\notag \\
& = 
\int
\left \langle \mathcal{G}_{\mathcal{F}}[X(t, \cdot)_{\#} \mu_{0}](X(t,x)) , \strut \,  X_{\ast}(x) - X(t,x) \right \rangle
\mathrm{d}\mu_{0}(x), 
\label{eq:diff_A_2}
\end{align}
where the second equality is owing to \eqref{eq:ODE_for_X}. 

It follows from \eqref{eq:diff_A_1} and \eqref{eq:diff_A_2} that 
\begin{align}
\frac{\mathrm{d}}{\mathrm{d}t} \mathcal{L}(t)
\leq
\mathcal{F}(\mu_{t}) - \mathcal{F}(\mu_{\ast})
+ \int
\langle \mathcal{G}_{\mathcal{F}}[X(t, \cdot)_{\#} \mu_{0}](X(t,x)) , \strut \, X_{\ast}(x) - X(t,x) \rangle
\, \mathrm{d}\mu_{0}(x)
\leq 0,
\notag
\end{align}
where 
the last inequality is owing to the {\newconvexity} of $\mathcal{F}$ in Definition~\ref{dfn:pf_convex}. 
As a result,
by the estimates in 
\eqref{eq:Lyap_grad_flow_upper_bound},
we can show that 
$\mathcal{F}(\mu_{t}) - \mathcal{F}(\mu_{\ast}) = \mathrm{O}(1/t)$. 
\end{proof}

\subsubsection{Proof of Theorem~\ref{thm:c_rate_disc_flow}}

To prove Theorem~\ref{thm:c_rate_disc_flow}, we use the Lyapunov function
\begin{align}
\mathcal{L}_{n} 
:=
n \gamma (\mathcal{F}(\hat{\mu}_{n}) - \mathcal{F}(\mu_{\ast}))
+
D(\hat{\mu}_{n}, \mu_{\ast}).
\label{eq:disc_grad_Lyap}
\end{align}
Our objective is to show that the difference $\mathcal{L}_{n+1} - \mathcal{L}_{n}$ is non-positive. 
We begin with the following calculation:
\begin{align}
\mathcal{L}_{n+1} - \mathcal{L}_{n}
= & \, 
(n+1) \gamma (\mathcal{F}(\hat{\mu}_{n+1}) - \mathcal{F}(\mu_{\ast}))
\notag \\
& \, 
- (n+1) \gamma (\mathcal{F}(\hat{\mu}_{n}) - \mathcal{F}(\mu_{\ast}))
+ \gamma (\mathcal{F}(\hat{\mu}_{n}) - \mathcal{F}(\mu_{\ast}))
\notag \\
& \, 
+ D(\hat{\mu}_{n+1}, \mu_{\ast}) - D(\hat{\mu}_{n}, \mu_{\ast})
\notag \\
= & \, 
(n+1) \gamma \, \mathcal{P}_{n} + \mathcal{Q}_{n},
\label{eq:diff_of_disc_Lyap} 
\end{align}
where
\begin{align}
& \mathcal{P}_{n} 
:= 
\mathcal{F}(\hat{\mu}_{n+1}) - \mathcal{F}(\hat{\mu}_{n}),
\label{eq:diff_of_disc_Lyap_first} \\
& \mathcal{Q}_{n} 
:= 
\gamma (\mathcal{F}(\hat{\mu}_{n}) - \mathcal{F}(\mu_{\ast}))
+ 
D(\hat{\mu}_{n+1}, \mu_{\ast}) - D(\hat{\mu}_{n}, \mu_{\ast}).
\label{eq:diff_of_disc_Lyap_second}
\end{align}

\begin{lem}
\label{thm:bd_hat_P_n}
On the assumptions of Theorem~\ref{thm:c_rate_disc_flow}, 
the value $\mathcal{P}_{n}$ in \eqref{eq:diff_of_disc_Lyap_first} is bounded as follows:
\begin{align}
\mathcal{P}_{n}
\leq
\left( - \gamma + \frac{L \gamma^{2}}{2} \right)
\int  
\left \|
\mathcal{G}_{\mathcal{F}}[\hat{\mu}_{n}](z)
\right \|^{2} \,
\mathrm{d} \hat{\mu}_{n}(z).
\label{eq:bound_hatP_n}
\end{align}
\end{lem}

\begin{proof}
By the {\newsmoothness} of $\mathcal{F}$ in Definition~\ref{dfn:pf_L_smooth}, 
we have
\begin{align}
\mathcal{P}_{n}
& = 
\mathcal{F}((X_{n+1})_{\#} \mu_{0}) - \mathcal{F}((X_{n})_{\#} \mu_{0}) 
\notag \\
& \leq 
\int  
\left \langle
\mathcal{G}_{\mathcal{F}}[(X_{n})_{\#}\mu_{0}](X_{n}(x)), \, X_{n+1}(x) - X_{n}(x)
\right \rangle \, 
\mathrm{d} \mu_{0}(x)
\notag \\
& \phantom{\leq} \
+ 
\frac{L}{2} \int 
\| X_{n+1}(x) - X_{n}(x) \|^{2} \, 
\mathrm{d} \mu_{0}(x)
\notag \\
& =
-\gamma
\int  
\left \|
\mathcal{G}_{\mathcal{F}}[(X_{n})_{\#}\mu_{0}](X_{n}(x))
\right \|^{2}
\mathrm{d} \mu_{0}(x)
+ 
\frac{L \gamma^{2}}{2} \int 
\left \|
\mathcal{G}_{\mathcal{F}}[(X_{n})_{\#}\mu_{0}](X_{n}(x))
\right \|^{2}
\mathrm{d} \mu_{0}(x)
\notag \\
& =
\left( - \gamma + \frac{L \gamma^{2}}{2} \right)
\int  
\left \|
\mathcal{G}_{\mathcal{F}}[\hat{\mu}_{n}](z)
\right \|^{2}
\mathrm{d} \hat{\mu}_{n}(z).
\notag
\end{align}
\end{proof}

\begin{lem}
\label{thm:bd_hat_Q_n}
On the assumptions of Theorem~\ref{thm:c_rate_disc_flow}, 
the value $\mathcal{Q}_{n}$ in \eqref{eq:diff_of_disc_Lyap_second} is bounded as follows:
\begin{align}
\mathcal{Q}_{n}
\leq 
\frac{\gamma^{2}}{2} \int \| \mathcal{G}_{\mathcal{F}}[\hat{\mu}_{n}](z) \|^{2} \mathrm{d} \hat{\mu}_{n}(z).
\label{eq:bound_hatQ_n}
\end{align}
\end{lem}

\begin{proof}
To estimate $D(\hat{\mu}_{n+1}, \mu_{\ast}) - D(\hat{\mu}_{n}, \mu_{\ast})$
in the RHS in \eqref{eq:diff_of_disc_Lyap_second}, 
we use the following relation%
\footnote{
This is derived from the following fundamental relation:
\begin{align}
& \| a - c \|^{2} = \| a - b \|^{2} + 2 \langle a-b, b-c \rangle + \| b - c \|^{2}
\notag \\
& \iff \| a - c \|^{2} - \| b - c \|^{2} = 2 \langle a-b, b-c \rangle + \| a - b \|^{2}.
\notag
\end{align}}:
\begin{align}
& \frac{1}{2} \| X_{n+1}(x) - X_{\ast}(x) \|^{2} - \frac{1}{2} \| X_{n}(x) - X_{\ast}(x) \|^{2} 
\notag \\
& = 
\langle X_{n+1}(x) - X_{n}(x), \, X_{n}(x) - X_{\ast}(x) \rangle + \frac{1}{2} \| X_{n+1}(x) - X_{n}(x) \|^{2}. 
\notag
\end{align}
By integrating both sides with respect to $\mathrm{d} \mu_{0}(x)$, 
we have
\begin{align}
D(\hat{\mu}_{n+1}, \mu_{\ast}) - D(\hat{\mu}_{n}, \mu_{\ast})
& = 
\gamma \int 
\langle -\mathcal{G}_{\mathcal{F}}[(X_{n})_{\#} \mu_{0}](X_{n}(x)), \, X_{n}(x) - X_{\ast}(x) \rangle 
\, \mathrm{d} \mu_{0}(x)
\notag \\
& \phantom{=} \, 
+ \frac{\gamma^{2}}{2} \int \| \mathcal{G}_{\mathcal{F}}[(X_{n})_{\#} \mu_{0}](X_{n}(x)) \|^{2} \, \mathrm{d} \mu_{0}(x) 
\notag \\
& = 
\gamma \int 
\langle \mathcal{G}_{\mathcal{F}}[(X_{n})_{\#} \mu_{0}](X_{n}(x)), \, X_{\ast}(x) - X_{n}(x) \rangle 
\, \mathrm{d} \mu_{0}(x) 
\notag \\
& \phantom{=} \, 
+ \frac{\gamma^{2}}{2} \int \| \mathcal{G}_{\mathcal{F}}[\hat{\mu}_{n}](z) \|^{2} 
\, \mathrm{d} \hat{\mu}_{n}(z). 
\notag
\end{align}
Therefore we have
\begin{align}
\mathcal{Q}_{n} 
& = 
\gamma 
\left( 
\mathcal{F}(\hat{\mu}_{n}) - \mathcal{F}(\mu_{\ast})
+ 
\int  
\langle \mathcal{G}_{\mathcal{F}}[(X_{n})_{\#} \mu_{0}](X_{n}(x)), \, X_{\ast}(x) - X_{n}(x) \rangle 
\, \mathrm{d} \mu_{0}(x) 
\right)
\notag \\
& \phantom{=} \, 
+ 
\frac{\gamma^{2}}{2} \int \| \mathcal{G}_{\mathcal{F}}[\hat{\mu}_{n}](z) \|^{2} \, \mathrm{d} \hat{\mu}_{n}(z)
\notag \\
& \leq 
\frac{\gamma^{2}}{2} \int \| \mathcal{G}_{\mathcal{F}}[\hat{\mu}_{n}](z) \|^{2} \, \mathrm{d} \hat{\mu}_{n}(z),
\notag
\end{align}
where this inequality is owing to the {\newconvexity} of $\mathcal{F}$ in Definition~\ref{dfn:pf_convex}. 
\end{proof}

\begin{proof}[Proof of Theorem~\ref{thm:c_rate_disc_flow}]
By \eqref{eq:diff_of_disc_Lyap}, \eqref{eq:bound_hatP_n} and \eqref{eq:bound_hatQ_n}, 
we have
\begin{align}
\mathcal{L}_{n+1} - \mathcal{L}_{n}
& \leq
\left[
(n+1) \left( -\gamma^{2} + \frac{L \gamma^{3}}{2} \right) + \frac{\gamma^{2}}{2}
\right] 
\int \| \mathcal{G}_{\mathcal{F}}[\hat{\mu}_{n}](z) \|^{2} \, \mathrm{d} \hat{\mu}_{n}(z)
\notag \\
& = 
-\left( n + \frac{1}{2} \right) \gamma^{2} \left( 1 - \frac{n+1}{2n+1} L \gamma \right)
\int \| \mathcal{G}_{\mathcal{F}}[\hat{\mu}_{n}](z) \|^{2} \, \mathrm{d} \hat{\mu}_{n}(z).
\notag
\end{align}
Therefore $\mathcal{L}_{n+1} - \mathcal{L}_{n} \leq 0$ holds for $n \geq 1$ if 
\(
\gamma \leq 3/(2 L)
\). 
Then, 
we can derive from \eqref{eq:disc_grad_Lyap} the following estimate:
\begin{align}
n \gamma (\mathcal{F}(\hat{\mu}_{n}) - \mathcal{F}(\mu_{\ast}))
\leq 
\mathcal{L}_{n} 
\leq 
\mathcal{L}_{1}
= 
\gamma (\mathcal{F}(\hat{\mu}_{1}) - \mathcal{F}(\mu_{\ast}))
+
D(\hat{\mu}_{1}, \mu_{\ast}).
\notag
\end{align}
Therefore we have the conclusion as follows:
\begin{align}
\mathcal{F}(\hat{\mu}_{n}) - \mathcal{F}(\mu_{\ast})
\leq
\frac{\gamma (\mathcal{F}(\hat{\mu}_{1}) - \mathcal{F}(\mu_{\ast}))
+
D(\hat{\mu}_{1}, \mu_{\ast})}{n \gamma}.
\notag
\end{align}
\end{proof}

\subsubsection{Proof of Theorem~\ref{thm:c_rate_acc_cont_flow}}
\label{sec:est_acc_grad_cont_flow}

\begin{proof}
We use the Lyapunov function
\begin{align}
\mathcal{L}^{\mathrm{ac}}(t)
:=
t^{r} (\mathcal{F}(\mu_{t}^{\mathrm{ac}}) - \mathcal{F}(\mu_{\ast}))
+
D(\nu_{t}^{\mathrm{ac}}, \mu_{\ast}). 
\end{align}
Then, we have
\begin{align}
\frac{\mathrm{d}}{\mathrm{d}t} \mathcal{L}^{\mathrm{ac}}(t) 
=
rt^{r-1} (\mathcal{F}(\mu_{t}^{\mathrm{ac}}) - \mathcal{F}(\mu_{\ast})) 
+ t^{r} \frac{\mathrm{d}}{\mathrm{d}t} \mathcal{F}(\mu_{t}^{\mathrm{ac}}) 
+ \frac{\mathrm{d}}{\mathrm{d}t} D(\nu_{t}^{\mathrm{ac}}, \mu_{\ast}).
\label{eq:acc_cont_Lyap_diff}
\end{align}
We calculate the second and third terms in the RHS above. 
First, we have
\begin{align}
t^{r} \frac{\mathrm{d}}{\mathrm{d}t} \mathcal{F}(\mu_{t}^{\mathrm{ac}})
& = 
t^{r} \int  
\left \langle \mathcal{G}_{\mathcal{F}}[\mu_{t}^{\mathrm{ac}}](X^{\mathrm{ac}}(t,x)), \frac{\partial}{\partial t} X^{\mathrm{ac}}(t,x) \right \rangle
\mathrm{d}\mu_{0}(x)
\notag \\
& = 
rt^{r-1} \int 
\Big \langle \mathcal{G}_{\mathcal{F}}[\mu_{t}^{\mathrm{ac}}](X^{\mathrm{ac}}(t,x)), Y^{\mathrm{ac}}(t,x) - X^{\mathrm{ac}}(t,x) \Big \rangle
\, \mathrm{d}\mu_{0}(x), 
\label{eq:func_Q}
\end{align}
where the first and second equalities are owing to 
\eqref{eq:char_vec_field_for_F_mu_t} and 
\eqref{eq:AccODE_for_XY_1_moment}, respectively. 
Next, we have
\begin{align}
\frac{\mathrm{d}}{\mathrm{d}t} D(\nu_{t}^{\mathrm{ac}}, \mu_{\ast})
& = 
\int  
\left \langle
\frac{\partial}{\partial t}Y^{\mathrm{ac}}(t,x), Y^{\mathrm{ac}}(t,x) - X_{\ast}(x)
\right \rangle \, 
\mathrm{d}\mu_{0}(x)
\notag \\
& = 
rt^{r-1}
\int 
\left \langle
-\mathcal{G}_{\mathcal{F}}[X^{\mathrm{ac}}(t,\cdot)_{\#} \mu_{0}](X^{\mathrm{ac}}(t,x)), \strut Y^{\mathrm{ac}}(t,x) - X_{\ast}(x)
\right \rangle \, 
\mathrm{d}\mu_{0}(x)
\notag \\
& = 
rt^{r-1}
\int 
\left \langle
\mathcal{G}_{\mathcal{F}}[\mu_{t}^{\mathrm{ac}}](X^{\mathrm{ac}}(t,x)), \strut X_{\ast}(x) - Y^{\mathrm{ac}}(t,x)
\right \rangle \, 
\mathrm{d}\mu_{0}(x),
\label{eq:func_S}
\end{align}
where the second equality is owing to \eqref{eq:AccODE_for_XY_2_vecfield}. 
By \eqref{eq:acc_cont_Lyap_diff}, \eqref{eq:func_Q}, and \eqref{eq:func_S}, 
the derivative of the Lyapunov function $\mathcal{L}^{\mathrm{ac}}(t)$
satisfies
\begin{align}
& \frac{\mathrm{d}}{\mathrm{d}t} \mathcal{L}^{\mathrm{ac}}(t) 
\notag \\
& =
rt^{r-1}
\left(
\mathcal{F}(\mu_{t}^{\mathrm{ac}}) - \mathcal{F}(\mu_{\ast})
+
\int 
\left \langle
\mathcal{G}_{\mathcal{F}}[\mu_{t}^{\mathrm{ac}}](X^{\mathrm{ac}}(t,x)), \strut X_{\ast}(x) - X^{\mathrm{ac}}(t,x)
\right \rangle \,
\mathrm{d}\mu_{0}(x)
\right)
\leq 0. 
\notag
\end{align}
The last inequality is owing to the {\newconvexity} of $\mathcal{F}$ in Definition~\ref{dfn:pf_convex}. 
Then, we have
\begin{align}
t^{r} \, (\mathcal{F}(\mu_{t}^{\mathrm{ac}}) - \mathcal{F}(\mu_{\ast}))
\leq 
\mathcal{L}^{\mathrm{ac}}(t) 
\leq 
\mathcal{L}^{\mathrm{ac}}(0) 
= 
D(\mu_{0}^{\mathrm{ac}}, \mu_{\ast}),
\end{align}
which implies
\begin{align}
\mathcal{F}(\mu_{t}^{\mathrm{ac}}) - \mathcal{F}(\mu_{\ast}) \leq \frac{1}{t^{r}} \, D(\mu_{0}^{\mathrm{ac}}, \mu_{\ast}).
\end{align}
This inequality means that 
$\mathcal{F}(\mu_{t}^{\mathrm{ac}}) - \mathcal{F}(\mu_{\ast})$ converges to $0$ with order $\mathrm{O}(1/t^{r})$
as $t \to \infty$.
\end{proof}

\subsubsection{Proof of Theorem~\ref{thm:c_rate_acc_disc_flow}}

To prove Theorem~\ref{thm:c_rate_acc_disc_flow}, 
we use the Lyapunov function
\begin{align}
\mathcal{L}_{n}^{\mathrm{ac}} 
:=
a_{n} (\mathcal{F}(\hat{\nu}_{n}^{\mathrm{ac}}) - \mathcal{F}(\mu_{\ast}))
+
D(\hat{\xi}_{n}^{\mathrm{ac}}, \mu_{\ast}).
\label{eq:disc_grad_Lyap_Acc}
\end{align}
We calculate the difference $\mathcal{L}_{n+1}^{\mathrm{ac}}  - \mathcal{L}_{n}^{\mathrm{ac}}$ as follows:
\begin{align}
\mathcal{L}_{n+1}^{\mathrm{ac}}  - \mathcal{L}_{n}^{\mathrm{ac}}
& = 
a_{n+1} (\mathcal{F}(\hat{\nu}_{n+1}^{\mathrm{ac}}) - \mathcal{F}(\mu_{\ast}))
- a_{n} (\mathcal{F}(\hat{\nu}_{n}^{\mathrm{ac}}) - \mathcal{F}(\mu_{\ast})) 
+ \mathcal{R}_{n}^{\mathrm{ac}}
\notag \\
& = 
a_{n+1} (\mathcal{F}(\hat{\nu}_{n+1}^{\mathrm{ac}}) - \mathcal{F}(\hat{\nu}_{n}^{\mathrm{ac}}) )
+ (a_{n+1} - a_{n}) (\mathcal{F}(\hat{\nu}_{n}^{\mathrm{ac}}) - \mathcal{F}(\mu_{\ast})) 
+ \mathcal{R}_{n}^{\mathrm{ac}}, 
\label{eq:diff_of_disc_Lyap_NAG} 
\end{align}
where
\begin{align}
\mathcal{R}_{n}^{\mathrm{ac}}
:=
D( \hat{\xi}_{n+1}^{\mathrm{ac}}, \mu_{\ast}) - D( \hat{\xi}_{n}^{\mathrm{ac}}, \mu_{\ast}).
\label{eq:diff_of_disc_Lyap_first_NAG}
\end{align}

\begin{lem}
\label{thm:expr_R_n_ac}
On the assumptions of Theorem~\ref{thm:c_rate_acc_disc_flow}, 
the value $\mathcal{R}_{n}^{\mathrm{ac}}$ in~\eqref{eq:diff_of_disc_Lyap_first_NAG}
is given by 
\begin{align}
\mathcal{R}_{n}^{\mathrm{ac}}
& = 
- \frac{(a_{n+1} - a_{n})^{2}}{2} \, \int \| \mathcal{G}_{\mathcal{F}}[\hat{\mu}_{n+1}^{\mathrm{ac}}](x) \|^{2} \, 
\mathrm{d}\hat{\mu}_{n+1}^{\mathrm{ac}}(x)
\notag \\
& \phantom{=} \, 
- (a_{n+1} - a_{n}) \int \langle \mathcal{G}_{\mathcal{F}}[\hat{\mu}_{n+1}^{\mathrm{ac}}](X_{n+1}^{\mathrm{ac}}(x)), Z_{n+1}^{\mathrm{ac}}(x) - X_{\ast}(x) \rangle \, 
\mathrm{d}\mu_{0}(x). 
\label{eq:diff_of_disc_Lyap_first_NAG_after}
\end{align}
\end{lem}

\begin{proof}
We consider the following identity%
\footnote{
This is derived from the following fundamental relation:
\begin{align}
& \| b - c \|^{2} 
= 
\| (a - b) - (a - c) \|^{2}
= 
\| a - b \|^{2} - 2 \langle a-b, a-c \rangle + \| a-c \|^{2}
\notag\\
& \iff
\| a - c \|^{2} - \| b - c \|^{2}
= 
- \| a - b \|^{2} + 2 \langle a-b, a-c \rangle.
\notag 
\end{align}
}: 
\begin{align}
& \| Z_{n+1}^{\mathrm{ac}}(x) - X_{\ast}(x) \|^{2} - \| Z_{n}^{\mathrm{ac}}(x) - X_{\ast}(x) \|^{2}
\notag \\
& = 
- \| Z_{n+1}^{\mathrm{ac}}(x) - Z_{n}^{\mathrm{ac}}(x) \|^{2} 
+ 2 \langle Z_{n+1}^{\mathrm{ac}}(x) - Z_{n}^{\mathrm{ac}}(x), Z_{n+1}^{\mathrm{ac}}(x) - X_{\ast}(x) \rangle. 
\notag
\end{align}
From this identity and \eqref{eq:acc_disc_grad_flow_for_Z}, we have
\begin{align}
& \| Z_{n+1}^{\mathrm{ac}}(x) -  X_{\ast}(x) \|^{2} - \| Z_{n}^{\mathrm{ac}}(x) - X_{\ast}(x) \|^{2}
\notag \\
& = 
- (a_{n+1} - a_{n})^{2} \, \| \mathcal{G}_{\mathcal{F}}[(X_{n+1}^{\mathrm{ac}})_{\#} \mu_{0}](X_{n+1}^{\mathrm{ac}}(x)) \|^{2} 
\notag \\
& \phantom{=} \, 
- 2 (a_{n+1} - a_{n}) \, \langle \mathcal{G}_{\mathcal{F}}[(X_{n+1}^{\mathrm{ac}})_{\#} \mu_{0}](X_{n+1}^{\mathrm{ac}}(x)), Z_{n+1}^{\mathrm{ac}}(x) - X_{\ast}(x) \rangle
\notag
\end{align}
By integrating both sides with respect to $\mathrm{d}\mu_{0}(x)$, 
we have \eqref{eq:diff_of_disc_Lyap_first_NAG_after}.
\end{proof}

From \eqref{eq:diff_of_disc_Lyap_NAG} and \eqref{eq:diff_of_disc_Lyap_first_NAG_after}, 
the difference $\mathcal{L}_{n+1}^{\mathrm{ac}}  - \mathcal{L}_{n}^{\mathrm{ac}}$ is expressed as
\begin{align}
\mathcal{L}_{n+1}^{\mathrm{ac}}  - \mathcal{L}_{n}^{\mathrm{ac}}
& = 
- \frac{(a_{n+1} - a_{n})^{2}}{2} \, \int \| \mathcal{G}_{\mathcal{F}}[\hat{\mu}_{n+1}^{\mathrm{ac}}](x) \|^{2} \, \mathrm{d}\hat{\mu}_{n+1}^{\mathrm{ac}}(x)
+ \mathcal{S}_{n}^{\mathrm{ac}}, 
\label{eq:diff_of_disc_Lyap_NAG_ver2}
\end{align}
where 
\begin{align}
& \mathcal{S}_{n}^{\mathrm{ac}}
:= 
a_{n+1} (\mathcal{F}(\hat{\nu}_{n+1}^{\mathrm{ac}}) - \mathcal{F}(\hat{\nu}_{n}^{\mathrm{ac}}) )
+ (a_{n+1} - a_{n}) \, \mathcal{T}_{n}^{\mathrm{ac}}, 
\label{eq:diff_of_disc_S_NAG} \\
& \mathcal{T}_{n}^{\mathrm{ac}}
:= 
\mathcal{F}(\hat{\nu}_{n}^{\mathrm{ac}}) - \mathcal{F}(\mu_{\ast}) - 
\int \langle \mathcal{G}_{\mathcal{F}}[\hat{\mu}_{n+1}^{\mathrm{ac}}](X_{n+1}^{\mathrm{ac}}(x)), Z_{n+1}^{\mathrm{ac}}(x) - X_{\ast}(x) \rangle \, 
\mathrm{d}\mu_{0}(x).
\label{eq:diff_of_disc_T_NAG}
\end{align}

\begin{lem}
\label{lem:bound_S_n_ac}
On the assumptions of Theorem~\ref{thm:c_rate_acc_disc_flow}, 
the value $\mathcal{S}_{n}^{\mathrm{ac}}$ in~\eqref{eq:diff_of_disc_S_NAG} is bounded as
\begin{align}
\mathcal{S}_{n}^{\mathrm{ac}}
& \leq
a_{n+1} \, (\mathcal{F}(\hat{\nu}_{n+1}^{\mathrm{ac}}) - \mathcal{F}(\hat{\mu}_{n+1}^{\mathrm{ac}}) )
+ 
(a_{n+1} - a_{n})^{2} \int \| \mathcal{G}_{\mathcal{F}}[\hat{\mu}_{n+1}^{\mathrm{ac}}](x) \|^{2} \, \mathrm{d}\hat{\mu}_{n+1}^{\mathrm{ac}}(x).
\label{eq:diff_of_disc_S_NAG_bound_3} 
\end{align}
\end{lem}

\begin{proof}
By the {\newconvexity} of $\mathcal{F}$ in Definition~\ref{dfn:pf_convex}, 
$\mathcal{T}_{n}^{\mathrm{ac}}$ is bounded as follows:
\begin{align}
\mathcal{T}_{n}^{\mathrm{ac}}
& = 
\mathcal{F}(\hat{\mu}_{n+1}^{\mathrm{ac}}) - \mathcal{F}(\mu_{\ast}) -
\int \langle \mathcal{G}_{\mathcal{F}}[\hat{\mu}_{n+1}^{\mathrm{ac}}](X_{n+1}^{\mathrm{ac}}(x)), X_{n+1}^{\mathrm{ac}}(x) - X_{\ast}(x) \rangle
\, \mathrm{d}\mu_{0}(x)
\notag \\
& \phantom{=} \, 
- \mathcal{F}(\hat{\mu}_{n+1}^{\mathrm{ac}}) + \mathcal{F}(\hat{\nu}_{n}^{\mathrm{ac}}) -
\int \langle \mathcal{G}_{\mathcal{F}}[\hat{\mu}_{n+1}^{\mathrm{ac}}](X_{n+1}^{\mathrm{ac}}(x)), Z_{n+1}^{\mathrm{ac}}(x) - X_{n+1}^{\mathrm{ac}}(x) \rangle
\, \mathrm{d}\mu_{0}(x)
\notag \\
& \leq 
- \mathcal{F}(\hat{\mu}_{n+1}^{\mathrm{ac}}) + \mathcal{F}(\hat{\nu}_{n}^{\mathrm{ac}}) - 
\int \langle \mathcal{G}_{\mathcal{F}}[\hat{\mu}_{n+1}^{\mathrm{ac}}](X_{n+1}^{\mathrm{ac}}(x)), Z_{n+1}^{\mathrm{ac}}(x) - X_{n+1}^{\mathrm{ac}}(x) \rangle
\, \mathrm{d}\mu_{0}(x). 
\label{eq:bound_hat_T_n}
\end{align}

Then, it follows from \eqref{eq:diff_of_disc_S_NAG} and \eqref{eq:bound_hat_T_n} that
\begin{align}
\mathcal{S}_{n}^{\mathrm{ac}}
& \leq 
a_{n+1} \, (\mathcal{F}(\hat{\nu}_{n+1}^{\mathrm{ac}}) - \mathcal{F}(\hat{\mu}_{n+1}^{\mathrm{ac}}) )
+ a_{n} \, (\mathcal{F}(\hat{\mu}_{n+1}^{\mathrm{ac}}) - \mathcal{F}(\hat{\nu}_{n}^{\mathrm{ac}}))
\notag \\
& \phantom{\leq} \, 
- (a_{n+1} - a_{n}) \int \langle \mathcal{G}_{\mathcal{F}}[\hat{\mu}_{n+1}^{\mathrm{ac}}](X_{n+1}^{\mathrm{ac}}(x)), Z_{n+1}^{\mathrm{ac}}(x) - X_{n+1}^{\mathrm{ac}}(x) \rangle
\, \mathrm{d}\mu_{0}(x)
\notag \\
& = 
a_{n+1} \, (\mathcal{F}(\hat{\nu}_{n+1}^{\mathrm{ac}}) - \mathcal{F}(\hat{\mu}_{n+1}^{\mathrm{ac}}) )
\notag \\
& \phantom{=} \, 
- a_{n+1} \int \langle \mathcal{G}_{\mathcal{F}}[\hat{\mu}_{n+1}^{\mathrm{ac}}](X_{n+1}^{\mathrm{ac}}(x)), Z_{n+1}^{\mathrm{ac}}(x) - X_{n+1}^{\mathrm{ac}}(x) \rangle
\, \mathrm{d}\mu_{0}(x) 
+ a_{n} \, \hat{\mathcal{U}}_{n}^{\mathrm{ac}}, 
\label{eq:diff_of_disc_S_NAG_bound_1} 
\end{align}
where
\begin{align}
\mathcal{U}_{n}^{\mathrm{ac}} 
:= 
\mathcal{F}(\hat{\mu}_{n+1}^{\mathrm{ac}}) - \mathcal{F}(\hat{\nu}_{n}^{\mathrm{ac}}) 
-
\int \langle \mathcal{G}_{\mathcal{F}}[\hat{\mu}_{n+1}^{\mathrm{ac}}](X_{n+1}^{\mathrm{ac}}(x)), X_{n+1}^{\mathrm{ac}}(x) - Z_{n+1}^{\mathrm{ac}}(x) \rangle
\, \mathrm{d}\mu_{0}(x). 
\label{eq:diff_of_disc_U_NAG}
\end{align}
By the {\newconvexity} of $\mathcal{F}$ in Definition~\ref{dfn:pf_convex}, 
$\hat{\mathcal{U}}_{n}^{\mathrm{ac}}$ is bounded as follows:
\begin{align}
\hat{\mathcal{U}}_{n}^{\mathrm{ac}} 
& = 
\mathcal{F}(\hat{\mu}_{n+1}^{\mathrm{ac}}) - \mathcal{F}(\hat{\nu}_{n}^{\mathrm{ac}}) 
-
\int \langle \mathcal{G}_{\mathcal{F}}[\hat{\mu}_{n+1}^{\mathrm{ac}}](X_{n+1}^{\mathrm{ac}}(x)), X_{n+1}^{\mathrm{ac}}(x) - Y_{n}^{\mathrm{ac}}(x) \rangle
\, \mathrm{d}\mu_{0}(x)
\notag \\
& \phantom{= \mathcal{F}(\hat{\mu}_{n+1}^{\mathrm{ac}}) - \mathcal{F}(\hat{\nu}_{n}^{\mathrm{ac}})} \, 
-
\int \langle \mathcal{G}_{\mathcal{F}}[\hat{\mu}_{n+1}^{\mathrm{ac}}](X_{n+1}^{\mathrm{ac}}(x)), Y_{n}^{\mathrm{ac}}(x) - Z_{n+1}^{\mathrm{ac}}(x) \rangle
\, \mathrm{d}\mu_{0}(x) 
\notag \\
& \leq
- \int \langle \mathcal{G}_{\mathcal{F}}[\hat{\mu}_{n+1}^{\mathrm{ac}}](X_{n+1}^{\mathrm{ac}}(x)), Y_{n}^{\mathrm{ac}}(x) - Z_{n+1}^{\mathrm{ac}}(x) \rangle
\, \mathrm{d}\mu_{0}(x). 
\label{eq:diff_of_disc_U_NAG_bound}
\end{align}

From \eqref{eq:diff_of_disc_S_NAG_bound_1} and \eqref{eq:diff_of_disc_U_NAG_bound}, 
we have
\begin{align}
\mathcal{S}_{n}^{\mathrm{ac}}
& \leq 
a_{n+1} \, (\mathcal{F}(\hat{\nu}_{n+1}^{\mathrm{ac}}) - \mathcal{F}(\hat{\mu}_{n+1}^{\mathrm{ac}}) )
- \int \langle \mathcal{G}_{\mathcal{F}}[\hat{\mu}_{n+1}^{\mathrm{ac}}](X_{n+1}^{\mathrm{ac}}(x)), \Omega_{n}^{\mathrm{ac}}(x) \rangle
\, \mathrm{d}\mu_{0}(x), 
\label{eq:diff_of_disc_S_NAG_bound_2} 
\end{align}
where
\begin{align}
\Omega_{n}^{\mathrm{ac}}(x) := a_{n+1}\, (Z_{n+1}^{\mathrm{ac}}(x) - X_{n+1}^{\mathrm{ac}}(x)) + a_{n}\, (Y_{n}^{\mathrm{ac}}(x) - Z_{n+1}^{\mathrm{ac}}(x)). 
\end{align}
By substituting $X_{n+1}^{\mathrm{ac}}$ in~\eqref{eq:acc_disc_grad_flow_for_X} to the RHS above, 
we obtain another expression of $\Omega_{n}^{\mathrm{ac}}(x)$ as follows:
\begin{align}
& \Omega_{n}^{\mathrm{ac}}(x) 
\notag \\
& = 
a_{n+1}\, \left( Z_{n+1}^{\mathrm{ac}}(x) - Y_{n}^{\mathrm{ac}}(x) - \frac{a_{n+1} 
- a_{n}}{a_{n+1}} (Z_{n}^{\mathrm{ac}}(x) - Y_{n}^{\mathrm{ac}}(x)) \right) + a_{n}\, (Y_{n}^{\mathrm{ac}}(x) - Z_{n+1}^{\mathrm{ac}}(x))
\notag \\
& = 
a_{n+1} Z_{n+1}^{\mathrm{ac}}(x) - a_{n+1} Y_{n}^{\mathrm{ac}}(x) - (a_{n+1} - a_{n}) Z_{n}^{\mathrm{ac}}(x) + (a_{n+1} - a_{n}) Y_{n}^{\mathrm{ac}}(x) 
\notag \\
& \phantom{=} \, 
+ a_{n} Y_{n}^{\mathrm{ac}}(x) - a_{n} Z_{n+1}^{\mathrm{ac}}(x) 
\notag \\
& = 
(a_{n+1} - a_{n}) (Z_{n+1}^{\mathrm{ac}}(x)  - Z_{n}^{\mathrm{ac}}(x) ). 
\label{eq:diff_of_disc_Omega_NAG_expr}
\end{align}

From \eqref{eq:diff_of_disc_S_NAG_bound_2} and \eqref{eq:diff_of_disc_Omega_NAG_expr}, 
we have the conclusion:
\begin{align}
\mathcal{S}_{n}^{\mathrm{ac}}
& \leq 
a_{n+1} \, (\mathcal{F}(\hat{\nu}_{n+1}^{\mathrm{ac}}) - \mathcal{F}(\hat{\mu}_{n+1}^{\mathrm{ac}}) )
\notag \\
& \phantom{\leq} \, 
- (a_{n+1} - a_{n}) \int \langle \mathcal{G}_{\mathcal{F}}[\hat{\mu}_{n+1}^{\mathrm{ac}}](X_{n+1}(x)), Z_{n+1}^{\mathrm{ac}}(x)  - Z_{n}^{\mathrm{ac}}(x) \rangle
\, \mathrm{d}\mu_{0}(x)
\notag \\
& = 
a_{n+1} \, (\mathcal{F}(\hat{\nu}_{n+1}^{\mathrm{ac}}) - \mathcal{F}(\hat{\mu}_{n+1}^{\mathrm{ac}}) )
\notag \\
& \phantom{\leq} \, 
- (a_{n+1} - a_{n})^{2} \int \| \mathcal{G}_{\mathcal{F}}[\hat{\mu}_{n+1}^{\mathrm{ac}}](X_{n+1}^{\mathrm{ac}}(x)) \|^{2} 
\, \mathrm{d}\mu_{0}(x)
\notag \\
& = 
a_{n+1} \, (\mathcal{F}(\hat{\nu}_{n+1}^{\mathrm{ac}}) - \mathcal{F}(\hat{\mu}_{n+1}^{\mathrm{ac}}) )
\notag \\
& \phantom{\leq} \, 
+ (a_{n+1} - a_{n})^{2} \int \| \mathcal{G}_{\mathcal{F}}[\hat{\mu}_{n+1}^{\mathrm{ac}}](x) \|^{2} 
\, \mathrm{d}\hat{\mu}_{n+1}^{\mathrm{ac}}(x),
\notag
\end{align}
where the first equality is owing to~\eqref{eq:acc_disc_grad_flow_for_Z}. 
\end{proof}

\begin{proof}[Proof of Theorem~\ref{thm:c_rate_acc_disc_flow}]
By combining~\eqref{eq:diff_of_disc_Lyap_NAG_ver2} and~\eqref{eq:diff_of_disc_S_NAG_bound_3}, 
we have
\begin{align}
\mathcal{L}_{n+1}^{\mathrm{ac}} - \mathcal{L}_{n}^{\mathrm{ac}}
& \leq
- \frac{(a_{n+1} - a_{n})^{2}}{2} \, \int \| \mathcal{G}_{\mathcal{F}}[\hat{\mu}_{n+1}^{\mathrm{ac}}](x) \|^{2} \, \mathrm{d}\hat{\mu}_{n+1}^{\mathrm{ac}}(x)
\notag \\
& \phantom{=} \, 
+ a_{n+1} \, (\mathcal{F}(\hat{\nu}_{n+1}^{\mathrm{ac}}) - \mathcal{F}(\hat{\mu}_{n+1}^{\mathrm{ac}}) )
\notag \\
& \phantom{\leq} \, 
+ (a_{n+1} - a_{n})^{2} \int \| \mathcal{G}_{\mathcal{F}}[\hat{\mu}_{n+1}^{\mathrm{ac}}](x) \|^{2}
\, \mathrm{d}\hat{\mu}_{n+1}^{\mathrm{ac}}(x)
\notag \\
& \leq
\frac{(a_{n+1} - a_{n})^{2}}{2} \, \int \| \mathcal{G}_{\mathcal{F}}[\hat{\mu}_{n+1}^{\mathrm{ac}}](x) \|^{2} 
\, \mathrm{d}\hat{\mu}_{n+1}^{\mathrm{ac}}(x)
\notag \\
& \phantom{=} \, 
+ a_{n+1} \, (\mathcal{F}(\hat{\nu}_{n+1}^{\mathrm{ac}}) - \mathcal{F}(\hat{\mu}_{n+1}^{\mathrm{ac}}) ). 
\label{eq:diff_of_disc_Lyap_NAG_ver3}
\end{align}

By~\eqref{eq:def_mathcal_G}, 
Formula \eqref{eq:acc_disc_grad_flow_for_Y} defining $Y_{n}^{\mathrm{ac}}$ is rewritten as
\begin{align}
Y_{n+1}^{\mathrm{ac}}(x) = X_{n+1}^{\mathrm{ac}}(x) - \frac{1}{L} \, \mathcal{G}_{\mathcal{F}}[\hat{\mu}_{n+1}^{\mathrm{ac}}](X_{n+1}^{\mathrm{ac}}(x)). 
\end{align}
Then, it follows from 
the {\newsmoothness} of $\mathcal{F}$ in Definition \ref{dfn:pf_L_smooth}
that
\begin{align}
& \mathcal{F}(\hat{\nu}_{n+1}^{\mathrm{ac}}) - \mathcal{F}(\hat{\mu}_{n+1}^{\mathrm{ac}}) 
\notag \\
& \leq
- \int \left\langle \mathcal{G}_{\mathcal{F}}[\hat{\mu}_{n+1}^{\mathrm{ac}}](x), \frac{1}{L} \, \mathcal{G}_{\mathcal{F}}[\hat{\mu}_{n+1}^{\mathrm{ac}}](x) \right\rangle
\, \mathrm{d}\hat{\mu}_{n+1}^{\mathrm{ac}}(x)
+ \frac{L}{2} \int \left\| \frac{1}{L} \, \mathcal{G}_{\mathcal{F}}[\hat{\mu}_{n+1}^{\mathrm{ac}}](x) \right\|^{2}
\, \mathrm{d}\hat{\mu}_{n+1}^{\mathrm{ac}}(x)
\notag \\
& = 
- \frac{1}{2L} \int \left\| \mathcal{G}_{\mathcal{F}}[\hat{\mu}_{n+1}^{\mathrm{ac}}](x) \right\|^{2}\, \mathrm{d}\hat{\mu}_{n+1}^{\mathrm{ac}}(x). 
\label{eq:ineq_by_L_smoothness}
\end{align}
By combining~\eqref{eq:diff_of_disc_Lyap_NAG_ver3} and~\eqref{eq:ineq_by_L_smoothness}, 
we have
\begin{align}
\mathcal{L}_{n+1}^{\mathrm{ac}} - \mathcal{L}_{n}^{\mathrm{ac}}
& \leq
\left(
\frac{(a_{n+1} - a_{n})^{2}}{2} - \frac{a_{n+1}}{2L}
\right)
\int \left\| \mathcal{G}_{\mathcal{F}}[\hat{\mu}_{n+1}^{\mathrm{ac}}](x) \right\|^{2}\, \mathrm{d}\hat{\mu}_{n+1}^{\mathrm{ac}}(x). 
\label{eq:diff_of_disc_Lyap_NAG_ver4}
\end{align}

Then, by~\eqref{eq:def_A_n}, we have
\begin{align}
a_{n+1} - a_{n} 
=
\frac{(n+2)(n+3) - (n+1)(n+2)}{16L}
=
\frac{n+2}{4L}. 
\notag
\end{align}
Therefore 
\begin{align}
(a_{n+1} - a_{n})^{2} - \frac{a_{n+1}}{L}
=
\frac{(n^{2} + 4n + 4) - (n^{2} + 5n + 6)}{16L^{2}}  
= 
- \frac{n+2}{16L^{2}}
\label{eq:coeff_given_by_A_n}
\end{align}
holds. 
From~\eqref{eq:diff_of_disc_Lyap_NAG_ver4} and~\eqref{eq:coeff_given_by_A_n}, 
we have
\begin{align}
\mathcal{L}_{n+1}^{\mathrm{ac}} - \mathcal{L}_{n}^{\mathrm{ac}}
& \leq
- \frac{n+2}{32 L^{2}}
\int \left\| \mathcal{G}_{\mathcal{F}}[\hat{\mu}_{n+1}^{\mathrm{ac}}](x) \right\|^{2} \mathrm{d}\hat{\mu}_{n+1}^{\mathrm{ac}}(x) 
\leq 0.
\label{eq:diff_of_disc_Lyap_NAG_ver5}
\end{align}

Finally, 
by recalling the definition of $\mathcal{L}_{n}^{\mathrm{ac}}$ in~\eqref{eq:disc_grad_Lyap_Acc},
we have
\begin{align}
a_{n} (\mathcal{F}(\hat{\nu}_{n}^{\mathrm{ac}}) - \mathcal{F}(\mu_{\ast}))
\leq 
\mathcal{L}_{n}^{\mathrm{ac}}
\leq 
\mathcal{L}_{1}^{\mathrm{ac}}
=
a_{1} (\mathcal{F}(\hat{\nu}_{1}^{\mathrm{ac}}) - \mathcal{F}(\mu_{\ast}))
+
D(\hat{\xi}_{1}^{\mathrm{ac}}, \mu_{\ast}),
\notag
\end{align}
which implies the conclusion. 
\end{proof}


\end{document}